\documentclass[10pt]{amsart}
\usepackage{color, graphicx, comment}
\usepackage{tikz}

\allowdisplaybreaks

\numberwithin{equation}{section} \overfullrule 5pt
\newtheorem{Theorem}{Theorem}[section]

\theoremstyle{definition}
\newtheorem{Definition}{Definition}[section]

\font\KFracFont=cmr12 at 20pt
\def\KKK{\mathop{\lower 4pt\hbox{\KFracFont K}}\limits}

\newcommand{\lastcfrac}[2]{%
  \vphantom{\cfrac{#1}{#2}}%
  \raisebox{\dimexpr1ex-\height}{%
    $\displaystyle
      \raisebox{.5\height}{$\ddots$}\cfrac{#1}{#2}
    $%
  }%
}

\def\cFrac#1#2{%
  \begin{array}{@{}c@{}}\multicolumn{1}{c|}{#1}\\%
  \hline\multicolumn{1}{|c}{#2}\end{array}\;}

\newcommand{\KK}{\mathbb{K}}

\DeclareMathOperator{\dd}{dd}
\DeclareMathOperator{\des}{des}
\DeclareMathOperator{\asc}{asc}
\DeclareMathOperator{\da}{da}
\DeclareMathOperator{\pk}{pk}
\DeclareMathOperator{\val}{val}

\usepackage{tikz}

\usepackage{epsf}

\usepackage[english]{babel}

\title[Hankel determinants of the Euler numbers]{Hankel Continued 
fractions and Hankel determinants of the Euler numbers}

\author{Guo-Niu HAN}
\address{I.R.M.A., UMR 7501, Universit\'e de Strasbourg
et CNRS, 7 rue Ren\'e Descartes, F-67084 Strasbourg, France}
\email{guoniu.han@unistra.fr}



\date{October 9, 2019}

\begin{document}

\begin{abstract} 
The Euler numbers occur in the Taylor expansion of $\tan(x)+\sec(x)$.
Since Stieltjes,
continued fractions and Hankel determinants of the even Euler numbers, on the one hand,
of the odd Euler numbers, on the other hand,
have been
widely studied separately.
However, no Hankel determinants of the (mixed) Euler numbers  have 
been obtained.
The reason for that is that some Hankel determinants of the 
Euler numbers are null. 
This implies that the Jacobi continued fraction of the 
Euler numbers does not exist. In the present paper, this obstacle 
is bypassed by using the Hankel continued fraction, instead of 
the $J$-fraction.  Consequently, 
an explicit formula for the Hankel determinants of the Euler numbers is being derived, as well as a full list of 
 Hankel continued fractions and Hankel determinants involving 
 Euler numbers. Finally,
a new $q$-analog of the Euler numbers $E_n(q)$ based on our continued fraction is proposed. We obtain an explicit formula for $E_n(-1)$
and prove a conjecture by R. J. Mathar on these numbers.

\end{abstract}

\maketitle

\section{Introduction} 
The {\it Euler numbers} $E_n (n\geq 0)$ are defined by their generating function
\begin{equation}\label{eq:EulerNumbers}
	\tan(x)+\sec(x)=\sum E_n \frac {x^n}{n!}.
\end{equation}
The even (resp. odd) Euler numbers $E_{2n}$ (resp. $E_{2n+1}$)
are also called {\it secant} (resp. {\it tangent}) numbers,
and their first values read:
\begin{equation*}
	\begin{array}{*{13}c}
		n      &=& 0 & 1 & 2 & 3 & 4 & 5 & 6 & 7 & 8 & 9   \\
		E_n  &  =& 1 & 1  &1  &2 & 5  &16& 61& 272 & 1385& 7936  \\
\end{array}%
\end{equation*}
As already proved back in 1879 by Andr\'e \cite{Andre1879},
the Euler numbers count the alternating permutations 
and
satisfy the following recurrence relation 
\begin{equation}\label{eq:En_rec}
	E_0=1;\  E_1=1;\quad
	E_n = \frac {1}{2} \sum_{k=1}^n \binom{n-1}{k-1} E_{k-1} E_{n-k}.
\end{equation}
The Euler numbers have been widely studied in Combinatorics (see \cite{Andre1879, Euler1755, Knuth1967Buckholtz, Nielsen1923, Viennot1982})
and are closely connected with  the Bernoulli and Genocchi numbers
\cite{Viennot1982, AlSalam1959Carlitz}.

\medskip

Stieltjes derived the continued fractions for the ordinary generating
functions of the tangent and secant numbers 
(see \cite{Stieltjes1894Re}, \cite{Flajolet1980}, 
\cite[p. 206, (53.11)]{Wall1948},
\cite[p. 369]{Wall1948}
)
\begin{align}	
	\sum_{n\geq 0} E_{2n} x^{2n} &= 
	\cFrac {1}{1} - \cFrac{1^2x^2}{1}
	- \cFrac{2^2x^2}{1}- \cFrac{3^2x^2}{1} - \cdots \label{eq:E2n} \\
	\sum_{n\geq 0} E_{2n+1} x^{2n+1} &= 
	\cFrac {x}{1} - \cFrac{1\cdot 2x^2}{1}
	- \cFrac{2\cdot 3x^2}{1}- \cFrac{3\cdot 4x^2}{1} - \cdots\label{eq:E2np1}
\end{align}
Thanks to the seminal paper by Flajolet \cite{Flajolet1980} on combinatorial aspects of
continued fractions, the above two continued fractions have 
become very classical, and been generalized
in several directions \cite{Prodinger2011, Fulmek2000, JosuatVerges2013Kim, Shin2010Zeng, Han1999RZeng, JosuatVerges2010}, including their $q$-analogs.

Although a lot of continued fractions 
involving either the tangent numbers $(E_{2n+1})$,
or secant numbers $(E_{2n})$ have been studied {\it separately},
the Jacobi continued fraction and the Hankel determinants of the ordinary generating function $\sum_{n\geq 0} E_n x^n$
for the {\it mixed} Euler numbers have never
been derived. 
The reason for that
is that {\it some} Hankel determinants of the Euler number sequence 
are null. 
This implies that the corresponding Jacobi continued fraction 
does 
not exist. In the present paper, this obstacle is bypassed by using the 
Hankel continued fraction \cite{Han2016Adv}, instead of the $J$-fraction.
Consequently, we derive an explicit formula for 
	the Hankel determinants of the Euler numbers. 
	See Section \ref{sec:ContFrac} for the basic definition and 
	properties of Hankel continued fractions.
We establish the following two theorems about the Hankel continued
fraction and the Hankel determinants of the Euler numbers. As expected,
we see that  some of the Hankel determinants are zero.

\begin{Theorem}\label{th:En_H}
	We have the following Hankel continued fraction expansion of the
	(mixed) Euler numbers:
\begin{align}
	\sum_{n\geq 0} E_n x^n
	&=
	 \cFrac{ 1}{ 1-x} 
	- \cFrac{ x^3}{1-2x-4x^2}
	- \cFrac{ 9x^3}{ 1-5x}
	- \cFrac{ 4x^2}{ 1-7x} \label{eq:En_H}  \\
	&\qquad\qquad
	- \cFrac{ 75x^3}{ 1-6x-36x^2}
	- \cFrac{ 147x^3}{1-11x}
	- \cFrac{ 16x^2}{1-13x} - \cdots \nonumber\\
	&= b_0 
	+ \cFrac{a_1}{b_1}
	+ \cFrac{a_2}{b_2}
	+ \cFrac{a_3}{b_3} + \cdots\nonumber
\end{align}
the general patterns for the coefficients $a_j$ and $b_j$ being: 
\begin{align*}
	a_1&=1 , \\
	a_{3k} &=-(4k - 1)^2(2k - 1)x^3,  \\
	a_{3k+1} &=-4k^2x^2 , \\
	a_{3k+2} &=-(4k + 1)^2(2k + 1)x^3 ; \\
b_0&=0 , \\
	b_{3k} &= -(6k - 1)x + 1 , \\
	b_{3k+1} &= -(6k + 1)x + 1,  \\
	b_{3k+2} &= -4(2k + 1)^2x^2 - 2(2k + 1)x + 1 .
\end{align*}
	\end{Theorem}

Throughout this paper we use the following convention, called ``{\it index convention}", 
saying that each expression given by cases is only valid for integers which have not been considered as previous special values.
In the above example,
the expression $b_{3k+1}$ is valid for $k\geq 0$; However, 
the expression $a_{3k+1}$ is valid for $k\geq 1$, but not for $k=0$, since $a_1$ is already listed above. 

\medskip

The {\it Hankel determinant} of order $n$
of the formal power series $f(x)=c_0+c_1x+c_2x^2+\cdots$ 
(or of the sequence $(c_0,c_1,c_2, \cdots$)  is defined by 
\begin{equation*}
	H_n(c_0, c_1, c_2, \ldots):=\det (c_{i+j})_{0\leq i,j\leq n-1}
\end{equation*}
for $n\geq 1$, and  $H_0(c_0, c_1, c_2, \ldots)=1$ if $n=0$.

\begin{Theorem}\label{th:En_Det}
	The Hankel determinants $H_n$ of the (mixed) Euler numbers 
	\begin{equation*}
	(E_0, E_1, E_2, E_3, \ldots)
	\end{equation*}
	are given by the following formulas:
\begin{align*}
	H_0&=1, \\
	H_{4k+1} &= (-1)^{k} \frac{(2k)!^2}{2^{4k(2k-1)}} \prod_{j=1}^{2k-1} (2j+1)!^4, \\
	H_{4k+2}&=0,\\
	H_{4k+3} &= (-1)^{k+1}\frac{(2k+1)!^2}{2^{4k(2k+1)}} \prod_{j=1}^{2k} (2j+1)!^4,\\
	H_{4k+4} &= (-1)^{k+1}\frac{(2k+1)!^2(4k+3)!^2}{2^{2(2k+1)^2}} \prod_{j=1}^{2k} (2j+1)!^4.
\end{align*}
\end{Theorem}

The proofs of Theorems \ref{th:En_H} and \ref{th:En_Det} are given in Section 
\ref{sec:proofs} by making use of the Flajolet continued fraction for 
permutation statistics \cite{Flajolet1980} and some combinatorial models for the Euler numbers described in Section \ref{sec:Comb}.  
We provide a large list of Hankel continued fractions
and Hankel determinants involving Euler numbers. 
The formulas involving the ordinary generating functions of Euler numbers,
most of them stated and proved in Section \ref{sec:ordgf}, 
are resumed in Table~\ref{tab:1}. 
In Section \ref{sec:expgf} we first obtain a continued fraction
of the exponential generating function of a quadruple permutation statistic,
see Theorem \ref{th:Exp_Fraction}.
Then, we derive the Hankel continued fractions and Hankel determinants
for the exponential generating functions involving
the Euler numbers. We resume these formulas in Table \ref{tab:2}.
There, we write $e_n=E_n/n!$ for short.
In the last section, a new $q$-analog of the Euler numbers $E_n(q)$ based on the continued fraction is proposed. We obtain an explicit formula for $E_n(-1)$
and prove a conjecture by R. J. Mathar on these numbers.
\smallskip

\begin{table}
	\caption{Formulas for the ordinary generating functions}\label{tab:1}
	$\begin{array}{|c|c|c|c|}
	\hline
	\text{Sequence} & \text{$H$-Fraction} & \text{Hankel det.}   & 
	              	\text{Exp. g. f.} \\
	\hline
	(E_0, E_1, E_2, E_3, E_4,  \ldots)  & \text{Thm\ref{th:En_H}}, H 
	      & \text{Thm\ref{th:En_Det}} & \tan(x)+\sec(x)\\
		(E_1, E_2, E_3, E_4, \ldots) & \text{(F7)}, J & \text{(H7)} & (\tan(x)+\sec(x))'\\
	(E_2, E_3, E_4,  \ldots)  & \text{(F10)}, H & \text{(H10)}& (\tan(x)+\sec(x))''\\
	(E_0, 0, E_2, 0, E_4, 0, \ldots)  & \text{(F1)}, J & \text{(H1)} & \sec(x)\\
	(0, E_2, 0, E_4, 0, \ldots)  & \text{(F8)}, H & \text{(H8)} &\sec(x)'\\
	(E_0, E_2, E_4, E_6, \ldots)  & \text{(F2)},J & \text{(H2)} & - \\
	(E_2, E_4, E_6,\ldots)  & \text{(F3)},J & \text{(H3)}  & - \\
	(0, E_1, 0, E_3, 0, E_5, \ldots)  & \text{(F9)}, H & \text{(H9)} & \tan(x) \\
	(E_1, 0, E_3, 0, E_5, \ldots)  & \text{(F4)}, J & \text{(H4)} & \tan(x)' \\
	(0, E_3, 0, E_5, \ldots)  & \text{(F11)},H & \text{(H11)} & \tan(x)'' \\
		(E_1, E_3,  E_5, E_7,\ldots)  & \text{(F5)},J & \text{(H5)} & - \\
	(E_3,  E_5, E_7, \ldots)  & \text{(F6)},J & \text{(H6)} & - \\
	\hline
\end{array}$
\end{table}
\begin{table}
	\caption{Formulas for the exponential generating functions}\label{tab:2}
	$\begin{array}{|c|c|c|c|}
	\hline
	\text{Sequence} & \text{$H$-Fraction} & \text{Hankel det.}   & 
	              	\text{Function} \\
	\hline
	(e_0, e_1, e_2, e_3, e_4, e_5,\ldots) & \text{(F17)}, J & \text{(H17)} & \tan(x)+\sec(x)\\
	(e_1, e_2, e_3, e_4, e_5, \ldots) & \text{(F16)}, J & \text{(H16)} & (\tan(x)+\sec(x)-1)/x\\
	(e_2, e_3, e_4, e_5, \ldots)  & \text{(F18)}, J & \text{(H18)}& \cdots \\
	(e_3, e_4, e_5, \ldots)  & \text{(F19)}, J & \text{(H19)}& \cdots \\
	(e_4, e_5, \ldots)  & \text{(F20)}, J & \text{(H20)}& \cdots \\
	(0, e_1, 0, e_3, 0, e_5, \ldots)  & \text{(F15)}, H & \text{(H15)} & \tan(x) \\
		(e_1, 0, e_3, 0, e_5, \ldots)  & \text{(F12)}, J & \text{(H12)} & \tan(x)/x \\
		(0, e_3, 0, e_5, \ldots)  & \text{(F14)}, H & \text{(H14)} &  (\tan(x)-x)/x^2 \\
		(e_3, 0, e_5, \ldots)  & \text{(F21)}, J & \text{(H21)} &  \cdots \\
		(e_1, e_3,  e_5, e_7,e_9,\ldots)  & \text{(F13)}, J & \text{(H13)} 
						 & \tan(\sqrt{x})/\sqrt{x}\\
		(e_3,  e_5, e_7,e_9, \ldots)  & \text{(F22)}, J & \text{(H22)} & (\tan(\sqrt{x})/\sqrt{x}-1)/x \\
	(e_5, e_7, e_9, \ldots)  & \text{(F23)}, J & \text{(H23)} & \cdots \\
	(e_7, e_9, \ldots)  & \text{(F24)}, J & \text{(H24)} & \cdots \\
	\hline
\end{array}$
\end{table}

Some further remarks are in order for a better understanding of our 
motivation, as well as various methods and notation used in this paper.

\smallskip

{\it Remark 1}. Some continued 
fractions obtained in the paper are of {\it Jacobi} type, 
as well as others need to be expressed as Hankel continued fractions. These two situations are indicated by the letters
``$J$" or ``$H$" in the second column in Tables~\ref{tab:1} and \ref{tab:2}. 

\smallskip

{\it Remark 2}. Some of these formulas are known or easy to prove.
We list them here for a quick view and comparison. In fact, we can find
(H1) in \cite{AlSalam1959Carlitz},
(H2) and (H3) for the case $r=1$ in \cite{AlSalam1959Carlitz, Radoux1992},  \cite[(3.52-53)]{Krattenthaler1998}   and \cite[(4.58-59)]{Milne2002},
(H22) and (H23)  in \cite{Morales2017PakII},
(F9) in \cite[(3.120)]{Milne2002}, [F16] in \cite[(3.9)]{Sokal2019EM}.

\smallskip

{\it  Remark 3}. The structure of the continued fractions makes 
no simple formula for the addition of two simple continued fractions.
In fact, our main result says that the sum of the two continued fractions
\eqref{eq:E2n} and \eqref{eq:E2np1} is equal to \eqref{eq:En_H}. This is not easy to prove since no addition formula is available.
In the same way, we know the continued fractions for $\tan(x)+\sec(x)$ and also $\tan(x)$, but we do not know the continued fraction for their difference $\sec(x)$.

\smallskip

{\it Remark 4}. There are simple continued fractions for
$ (E_0, 0, E_2, 0, E_4, \ldots)$
and
$ (E_0, E_2, E_4,  \ldots)$. However,
no simple continued fractions for 
$ (e_0, 0, e_2, 0, e_4, \ldots)$
and
$ (e_0, e_2, e_4,  \ldots)$ are known.
We are convinced that there is no unified method to derive continued fractions for 
ordinary generating functions and exponential generating functions.

\smallskip

{\it Remark 5}. For the exponential generating functions, the three 
families in Table~\ref{tab:2} 
seem to be naturally extended.  
However, for the ordinary generating functions, the five families 
cannot grow,
and all continued fractions of simple form are listed in 
Table~\ref{tab:1}.

\smallskip

{\it Remark 6}. In view of Remark 5, we may think that the exponential generating functions are more adequate for the Hankel continued fractions. However,
Remark~4 says the converse. 

\smallskip

{\it Remark 7}. Theorem \ref{th:En_H} and (F10) are very similar. 
However there is an important difference.
The super $1$-fraction of $\sum E_n x^n$ exists. This fact leads us
to find the proof of Theorem \ref{th:En_H}.
But there is no simple super $1$-fraction of $\sum E_{n+2} x^n$. 
Fortunately, now we have the proof of Theorem \ref{th:En_H} which gives us some indication for the proof of (F10).


\section{Definitions and properties of the Continued fractions} \label{sec:ContFrac}
In this section we recall some basic definitions and properties of
the general continued fractions and the super continued fractions,
including the Hankel continued fractions.

Let $\mathbb{K}$ be a field. In most cases, $\mathbb{K}$ 
will be the field  $\mathbb{Q}$ of rational numbers. 
Consider the field of fractional fractions $\mathbb{K}(x)$.
Let $\mathbf{a}=(a_1, a_2, \ldots)$ and
$\mathbf{b}=(b_0, b_1, b_2, \ldots)$ be two sequences of $\KK(x)$.
The {\it generalized continued fraction} associated with the two sequences $\mathbf{a}$ and $\mathbf{b}$ can be written by using the natural notation 
\begin{equation*}
	b_0 + \cfrac{a_1}{
		b_1 + \cfrac{a_2}{
			b_2 + \cfrac{a_3}{
				b_3 
					+ \lastcfrac{}{}
			}}
		}
\end{equation*}
or Pringsheim's notation:
\begin{equation}\label{eq:GCF}
b_0 +  \cFrac{a_1}{b_1} 
+ \cFrac{a_2}{b_2} 
+ \cFrac{a_3}{ b_3} 
+ {\cdots}
\end{equation}
For the notion about {\it value, partial numerator and denominator, fundamental recurrence formulas, 
equivalence transformations} of continued fractions, see \cite[p. 13-19]{Wall1948}.
The {\it value} of the above continued fraction is a formal power 
series in $x$ with coefficients in $\mathbb K$. 
Throughout the paper
we will extensively use 
the following contraction formulas.
The contractions with for the case $b_j=1$ are well-known \cite{Wall1948, Perron1957II, Stieltjes1894Re}.

\begin{Theorem}
The generalized continued fraction defined
	in \eqref{eq:GCF} has the following contraction formulas. 

	(1) Even contraction:
	\begin{align}
		&b_0 +  \cFrac{a_1b_2}{b_1b_2+a_2} 
-  \cFrac{a_2a_3b_4}{b_2b_3b_4+a_4b_2+a_3b_4} 
-  \cFrac{a_4a_5b_2b_6}{b_4b_5b_6+a_6b_4+a_5b_6} 
- {\cdots}\label{eq:EvenContraction}\\
		&=  b_0' + \cFrac {a_1'}{b_1'} 
		+ \cFrac{a_2'}{b_2'} 
		+ \cFrac{a_3'}{b_3'} 
		+ \cdots \nonumber 
\end{align}
The general patterns for the new coefficients $a_j'$ and $b_j'$ are:
\begin{align*}
	a_1' &=  a_1 b_2,   \cr
	a_2' &=  - a_{2}a_{3} { b_4}, \cr
	a_j' &=  - a_{2j-2} a_{2j-1} {b_{2j-4}} {  b_{2j}};\cr
	b_0'  &=  b_0, \cr
	b_1' &=  b_1b_2+a_2, \\
	b_j' &=   b_{2j-2} b_{2j-1} b_{2j} + a_{2j} b_{2j-2} +  a_{2j-1} b_{2j}.
\end{align*}
When $b_0=0$ and $b_1=b_2=\cdots=1$, the even contraction formula becomes
\begin{equation}
\cFrac{a_1}{1+a_2} 
-  \cFrac{a_2a_3}{1+a_3+a_4} 
-  \cFrac{a_4a_5}{1+a_5+a_6} 
- {\cdots}
\end{equation}

	(2) Odd contraction:
\begin{equation}
	\frac{b_0b_1+a_1}{b_1} 
	- \cFrac{a_1a_2b_3/b_1}{b_1b_2b_3+a_3b_1+a_2b_3} 
-  \cFrac{a_3a_4b_1b_5}{b_3b_4b_5+a_5b_3+a_4b_5} 
- {\cdots}
\end{equation}
The general patterns for the new coefficients $a_j'$ and $b_j'$ are:
\begin{align*}
	a_1' &= - a_1a_2 \frac {b_{3}}{b_{1}},   \cr
	a_2' &=  - a_{3}a_{4} { b_{1}  b_{5}},\cr
	a_j' &=  - a_{2j-1}a_{2j} {b_{2j-3}} {  b_{2j+1}}; \cr
	b_0'  &=  \frac {b_0b_1+a_1}{b_1},  \cr
	b_1' &=  b_1b_2b_3 + a_3b_1 + a_2b_3, \\
	b_j' &=   b_{2j-1} b_{2j} b_{2j+1} + a_{2j+1} b_{2j-1} +  a_{2j} b_{2j+1}.
\end{align*}
When $b_0=0$ and $b_1=b_2=\cdots=1$, the odd contraction formula becomes
\begin{equation}
	{a_1}
	- \cFrac{a_1a_2}{1+a_2+a_3} 
-  \cFrac{a_3a_4}{1+a_4+a_5} 
- {\cdots}
\end{equation}

	(3) Chop contraction:
	\begin{equation}\label{eq:chop}
\frac{b_0b_1+a_1}{b_1} - \cFrac{a_1a_2/b_1}{b_1b_2+a_2} +
\cFrac{a_3b_1}{b_3} + \cFrac{a_4}{b_4} +  \cFrac{a_5}{b_5} \cdots
\end{equation}

	(4) Haircut contraction: for each $\alpha\not= a_1/b_1$,
\begin{equation}\label{eq:Haircut}
	\alpha + \cFrac{a_1-\alpha b_1}{b_1} 
	+ \cFrac{a_1a_2}{b_2a_1-b_1b_2\alpha -a_2\alpha} +
	\cFrac{a_3(a_1-b_1\alpha)}{b_3} 
+	\cFrac{a_4}{b_4} 
	\cdots
\end{equation}

\end{Theorem}
\medskip

The even and odd contraction formulas are very classical (see 
\cite[p. 21]{Wall1948}, 
\cite[p. 12-13]{Perron1957II}, \cite[p. J3]{Stieltjes1894Re}). The chop and haircut formulas
can be verified directly. For the most general
contraction and extension, 
see \cite[p. 10-16]{Perron1957II}. \footnote{\; There were some errors 
in the first edition of the book by Perron \cite[p. 199]{Perron1913}, 
when the author derives the formula for
the general contraction. 
These errors
had been fixed in the second edition \cite[p. 11]{Perron1957II}}

\medskip

Let $\mathbf{u}=(u_1, u_2, \ldots)$ and $\mathbf{v}=(v_0, v_1, v_2, \ldots)$
be two sequences.
Recall that the {\it Jacobi continued fraction} attached to $(\mathbf{u}, \mathbf{v})$, 
or {\it $J$-fraction}, for short, is a continued
fraction of the form
\begin{equation}\label{eq:Jacobi}
f(x)=
	\cFrac{v_0 }{ 1 + u_1 x} 
	- \cFrac{v_1 x^2 }{ 1+u_2x} 
	- \cFrac{v_2 x^2 }{ 1 + u_3x} 
	- \cdots
\end{equation}
The basic properties on $J$-fractions can be found in
\cite{Krattenthaler1998, Krattenthaler2005, Flajolet1980, Wall1948, Stieltjes1894Re, Heilermann1846, Gessel2006Xin}. We emphasize the fact that
the Hankel determinants can be calculated
from the $J$-fraction by means of the following  fundamental relation, first
stated by Heilermann in 1846 \cite{Heilermann1846}:
\begin{equation}\label{eq:detH}
H_n(f)
= v_0^n v_1^{n-1} v_2^{n-2} \cdots v_{n-2}^2 v_{n-1}. 
\end{equation}

The Hankel determinants of a power series $f$ can be calculated by the 
above
fundamental relation 
if the $J$-fraction exists, which is equivalent to the fact that all Hankel determinants
of $f$
are nonzero.
If some of the Hankel determinants are zero,
we must use the
{\it Hankel continued fraction} ({\it $H$-fraction}, for short)
whose existence and uniqueness are guaranteed without any condition for the power series.
The Hankel determinants can also be evaluated by using the Hankel continued fraction. Let us recall the basic definition and properties of the Hankel continued fractions~\cite{Han2016Adv}.

\begin{Definition}\label{def:super}
For each positive integer $\delta$, a {\it super continued fraction} associated with $\delta$, called {\it super $\delta$-fraction} for short, is defined to be a continued fraction of the following form
\begin{equation}\label{eq:super}
F(x)=
\cFrac{v_0 x^{k_0}}{ 1+u_1(x) x} 
- \cFrac{ v_1 x^{k_0+k_1+\delta}}{ 1+u_2(x)x}
- \cFrac{ v_2 x^{k_1+k_2+\delta}}{ 1+u_3(x)x}
- {\cdots} 
\end{equation}
where $v_j\not=0$ are constants, $k_j$ are nonnegative integers and $u_j(x)$ are polynomials of 
degree less than or equal to $k_{j-1}+\delta-2$. By convention, $0$ is of degree $-1$.
\end{Definition}

When $\delta=1$ (resp. $\delta=2$) and all $k_j=0$, 
the super $\delta$-fraction \eqref{eq:super} 
is the traditional $S$-fraction (resp. $J$-fraction). 
A super $2$-fraction is called {\it Hankel continued fraction}.

\begin{Theorem}\label{th:super2}
(i) Let $\delta$ be a positive integer.
Each super $\delta$-fraction defines
a power series, and conversely, for each power series $F(x)$,
the super $\delta$-fraction expansion of $F(x)$ exists
and is unique.
\smallskip
(ii) Let $F(x)$ be a power series such that its $H$-fraction 
	is given by \eqref{eq:super} with $\delta=2$. 
Then, all non-vanishing Hankel determinants of $F(x)$ are given by
	\begin{equation}\label{eq:HankelDetFundamental}
		H_{s_j}(F(x))= (-1)^{\epsilon_j} v_0^{s_j} v_1^{s_j-s_1} v_2^{s_j-s_2} \cdots v_{j-1}^{s_j-s_{j-1}},
	\end{equation}
where $\epsilon_j = \sum_{i=0}^{j-1} {k_i(k_i+1)/2}$
and
$s_j=k_0+k_1+\cdots + k_{j-1}+j$ for every $j\geq 0$.
\end{Theorem}

See \cite{Han2016Adv, Roblet1994, Buslaev2010, Holtz2012T} for the proof of Theorem \ref{th:super2}.

\medskip

{\it A short historical remark}. The main idea of the determinant formula \eqref{eq:HankelDetFundamental} may already be known by Magnus in 1970 \cite{Magnus1970}.
The present form appeared for the first time in the Ph.D. thesis 
by Emmanuel Roblet in 1994 \cite[p. 44-56]{Roblet1994} \footnote{\ The author is grateful to 
Xavier Viennot, who gave me the reference
to the Roblet's thesis.}.  
It was independently rediscovered by Buslaev in 2010~\cite{Buslaev2010},
 by Boltz-Tyaglov in 2012~\cite{Holtz2012T}, and by the author in 2016~\cite{Han2016Adv}.
The name of the continued fraction called by Roblet and Buslaev is $P$-fraction.
The $P$-fraction ({\it ``principal part plus" fraction}) was introduced by Magnus in 1962 \cite{Magnus1962C, Magnus1962P, Magnus1970, Jones1980T}, also independently by Mills-Robbins in 1986 \cite{Mills1986R} and by  Boltz-Tyaglov in 2012~\cite{Holtz2012T}. Notice that there are slight differences
among the $P$-fraction used by Roblet \cite{Roblet1994}, the $P$-fraction used by Buslaev \cite{Buslaev2010}, and the Hankel continued fraction used in 
\cite{Han2016Adv}. Roblet's $P$-fraction is faithful to the Magnus's up to some equivalent transformations, and has the following form \cite[(2.9)]{Roblet1994}:
\begin{equation*}
	F(x)=	f_0+f_1x+f_2x^2+\cdots = 
	f_0+ \cFrac{\lambda_1 x^{N_1}}{ P_1(x)} 
- \cFrac{ \lambda_2 x^{N_1+N_2}}{ P_2(x)}
- \cFrac{ \lambda_3 x^{N_2+N_3}}{P_3(x)}
- {\cdots} 
\end{equation*}
Because of the presence of the constant term $f_0$ in the 
above fraction, 
it is necessary to consider the Hankel determinants
of $(f_1, f_2, f_3, \ldots)$ under the 
Magnus $P$-fraction notation rather than
of $(f_0, f_1, f_2, \ldots)$ \cite[(2.14), (2.18)]{Roblet1994}. 
On the other hand, Buslaev used the $P$-fraction notion without 
quoting
Magnus's papers. Moreover, the condition $N_j\geq 1$, needed in the original $P$-fraction notation, had been removed~\cite{Buslaev2010}.

A special case of \eqref{eq:HankelDetFundamental}
for a restricted family of $C$-fractions
was obtained by Scott-Wall in 1940 \cite{Scott1940Wall} and  independently by Cigler in 2013 \cite{Cigler2013spe}. It is interesting to notice that almost 
all these studies were of theoretical nature with no explicit examples, 
except some artificial ones in Cigler's paper. Finally, note that several 
real-life expansions and applications are given 
in~\cite{Han2016Adv}, together with an analog of Euler-Lagrange theorem 
about periodic continued fraction
for power series over a finite field.

\section{Combinatorial interpretations of the Euler numbers} \label{sec:Comb} 


It is well-known that the Euler numbers count the 
alternating permutations~\cite{Andre1879}.
The proof of our main Theorem \ref{th:En_H} leads us to derive
further combinatorial interpretations of the Euler numbers.

Given a permutation $\sigma=\sigma_1\sigma_2\cdots \sigma_n$ with
the convention $\sigma_0=\sigma_{n+1}=+\infty$.
For each $j\in \{1,2,\ldots, n\}$, the letter $\sigma_j$ is called
{\it descent} if $\sigma_j > \sigma_{j+1}$;
{\it ascent} if $\sigma_j < \sigma_{j+1}$;
{\it peak} if $\sigma_{j-1} < \sigma_j > \sigma_{j+1}$;
{\it valley} if $\sigma_{j-1} > \sigma_j < \sigma_{j+1}$;
{\it double ascent} if $\sigma_{j-1} < \sigma_j < \sigma_{j+1}$;
{\it double descent} if $\sigma_{j-1} > \sigma_j > \sigma_{j+1}$
(See \cite{Flajolet1980, Pan2019Zeng, Fu2018Amy, Zeng1993, Zhuang2017} or \cite[Exercise 1.61]{Stanley2012EC1}).
Let $\des(\sigma), \val(\sigma), \pk(\sigma), \da(\sigma), \dd(\sigma)$ be the
numbers of descents, valleys, peaks, double ascents and double descents of $\sigma$.
In 1974, Carlitz and Scoville obtained the exponential generating function 
of the
quadruple statistic $(\val, \pk, \da, \dd)$ for the permutations \cite{Carlitz1974Scoville, Fu2018Amy, Pan2019Zeng}.
\begin{equation}\label{eq:CarlitzScoville}
	\sum_{n\geq 1} \frac {x^n}{n!} \sum_{\sigma \in \mathfrak{S}_n}
	u_1^{\val(\sigma)} u_2^{\pk(\sigma)} u_3^{\da(\sigma)} u_4^{\dd(\sigma)}
	=u_1 \frac {e^{\alpha_2 x} - e^{\alpha_1 x}}{\alpha_2 e^{\alpha_1 x} - \alpha_1 e^{\alpha_2 x}},
\end{equation}
where $u_3+u_4=\alpha_1+\alpha_2$ and $u_1u_2=\alpha_1 \alpha_2$.

\begin{Definition}\label{def:WWWW}
	We define four weight functions $W_1(\sigma), W_2(\sigma), W_3(\sigma), W_4(\sigma)$ for the permutations $\sigma=\sigma_1\sigma_2\ldots \sigma_n$.
\begin{align}
	W_1(\sigma) &= \prod_{j=2}^n (-1)^{\chi(\text{$j$ is a double ascent})}, \\
	W_2(\sigma) &= \prod_{j=2}^n 
	0^{\chi(\text{$j$ is a double ascent})}
	\left(\frac {1}{2}\right)^{\chi(\text{$j$ is a peak})} ,\label{eq:defW2} \\
	W_3(\sigma) &= \prod_{j=2}^{n}  \left(\frac {1}{2}\right)^{\chi(\text{$j$ is not a peak})}, \\
	W_4(\sigma) &= \prod_{j=1}^{n-1} 
	\left(\frac{1+i}{2}\right)^{\chi(\text{$j$ is a ascent})}
	\left(\frac{1-i}{2}\right)^{\chi(\text{$j$ is a descent})},
\end{align}
where $i$ is the imaginary unit.
\end{Definition}

\begin{Theorem}\label{th:EnComb}
For each positive integer $n$	we have
	\begin{equation}	
	\sum_{\sigma \in \mathfrak{S}_n} W_2(\sigma)
	=\sum_{\sigma \in \mathfrak{S}_n} W_3(\sigma)
	=\sum_{\sigma \in \mathfrak{S}_n} W_4(\sigma)
	=E_n,
	\end{equation}
	and
	\begin{equation}\label{eq:W1}
\sum_{\sigma \in \mathfrak{S}_n} W_1(\sigma)
		= \begin{cases}
			E_n; &\text{ if $n$ is odd}\\
			0  . &\text{if $n$ is even}
	\end{cases}
\end{equation}
\end{Theorem}
Identity \eqref{eq:W1} is a result obtained by J. O. Tirrell and Y. Zhuang recently \cite{Tirrell2019Zhuang}.
For example, with $n=3$ and $E_n=2$, the 6 permutations with their four weights
are listed below. We check that Theorem \ref{th:EnComb} is true for $n=3$.
Notice that although the four results are the same, the four summations themselves are quite different.
\begin{equation*}
	\begin{array}{|c|c|c|c|c|}
		\hline
		\sigma      & W_1 & W_2 & W_3 & W_4  \\
		\hline 
		123  & (-1)\cdot(-1) &  0\cdot 0        & \frac 12\cdot \frac 12 & \frac{1+i}{2}\cdot \frac{1+i}{2}  \\ [3pt]
132  &  1\cdot 1     & \frac 12\cdot 1  & 1\cdot \frac 12 &  \frac{1+i}{2}\cdot \frac{1-i}{2}  \\[3pt]
213  &  1\cdot (-1)  &  1\cdot 0        & \frac 12\cdot \frac 12 &   \frac{1-i}{2}\cdot \frac{1+i}{2} \\[3pt]
231  &  1\cdot 1     & \frac 12\cdot 1  & 1\cdot \frac 12 &  \frac{1+i}{2}\cdot \frac{1-i}{2}  \\[3pt]
312  & 1\cdot (-1)   & 1\cdot 0         & \frac 12\cdot \frac 12 &  \frac{1-i}{2}\cdot \frac{1+i}{2}  \\[3pt]
321  & 1\cdot 1      &   1\cdot 1       & \frac 12\cdot \frac 12 &  \frac{1-i}{2}\cdot \frac{1-i}{2} \\ [3pt]
	\hline
		\text{sum} & 2 & 2 & 2 & 2 \\
	\hline
\end{array}%
\end{equation*}

It is well-known \cite{Foata1970Schutzenberger, Stanley2012EC1, Pan2019Zeng} that the exponential generating function of
the Eulerian polynomials
\begin{equation}\label{eq:An}
A_n(t)=\sum_{\sigma \in \mathfrak{S}_n} t^{1+\des(\sigma)}
\end{equation}
is 
\begin{equation}\label{eq:An_gf}
	1+\sum_{n\geq 1} A_n(t) \frac {x^n}{n!} = \frac {1-t}{1-te^{(1-t)x}}.
\end{equation}
Theorem \ref{th:EnComb} can be proved by using  \eqref{eq:An_gf} and
the Carlitz-Scoville formula without
difficulty. 
Let  $P_n(t,s)$ be the ordinary generating function of peaks and double ascents for the permutations:
\begin{equation}\label{eq:Pn}
	P_n(t,s)= \sum_{\sigma \in \mathfrak{S}_n} t^{\pk(\sigma)} s^{\da(\sigma)}.
\end{equation}
From the Carlitz-Scoville (\cite{Carlitz1974Scoville}, see also \cite{Fu2018Amy}), we have
\begin{equation}\label{eq:Pn_gf}
	P(x; t,s)=\sum_{n\geq 1} P_n(t,s) \frac {x^n}{n!}  
	=\frac {-2+2 e^{u(t,s)x}}{(1+s+u(t,s)) -(1+s-u(t,s)) e^{u(t,s)x}},
\end{equation}
where $u(t,s)=\sqrt{(1+s)^2-4t}$.

\begin{proof}[Proof of Theorem \ref{th:EnComb}]
	We have the following specializations.

	(1) When $t=1$ and $s=-1$, $u(1,-1)=2i$. Identity \eqref{eq:Pn_gf} becomes
	\begin{align*}
		P(x; 1,-1) 
		&=\frac {-2+2 e^{2xi}}{2i + 2i e^{2xi}} = \tan(x).
	\end{align*}
	Thus, relation \eqref{eq:W1} is true. 

	(2) When $t=1/2$ and $s=0$, $u(1/2, 0)=i$. Identity \eqref{eq:Pn_gf} becomes
	\begin{equation*}
		P(x; 1/2, 0) 
		=\frac {-2+2 e^{xi}}{(1+i) - (1-i) e^{xi}} = \tan(x)+\sec(x)-1.
	\end{equation*}
	Thus
	\begin{equation*}
		\tan(x)+\sec(x) = P(x;1/2,0)+1
	\end{equation*}
	or
	\begin{equation*}
		E_n = P_n(1/2,0). \qquad (n\geq 1)
	\end{equation*}

	(3) When $t=2$ and $s=1$, $u(2,1)=2i$. Identity \eqref{eq:Pn_gf} becomes
	\begin{align*}
		P(x; 2,1) 
		&=\frac {-2+2 e^{2xi}}{(2+2i) - (2- 2i) e^{2xi}} = 
		  \frac{\tan(2x)+\sec(2x)-1}{2}.
	\end{align*}
	Thus
	\begin{equation*}
		\tan(x)+\sec(x) = 2P(x/2;2,1)+1
	\end{equation*}
	or
	\begin{equation*}
		E_n = 2^{1-n} P_n(2,1). \qquad (n\geq 1)
	\end{equation*}

	(4) When $t=i$. Identity \eqref{eq:An_gf} becomes
	\begin{align*}
		1+\sum_{n\geq 1} A_n(i) \frac {x^n}{n!} 
		= \frac {1-i}{1-ie^{(1-i)x}} 
		= \frac{\tan((1+i)x)+\sec((1+i)x)-i}{1-i}.
	\end{align*}
	Thus
	\begin{equation*}
		\tan(x)+\sec(x) = 1+\sum_{n\geq 1}  \frac {A_n(i)}{(1+i)^{n-1}} \frac {x^n}{n!}
	\end{equation*}
	or
	\begin{equation*}
		E_n = -i(1+i)^{1-n} A_n(i). \qquad (n\geq 1)
	\end{equation*}
So that	
	\begin{equation*}
		E_n = \sum_{\sigma\in\mathfrak{S}_n} 
		\left(\frac{1+i}{2}\right)^{\asc(\sigma)}
		\left(\frac{1-i}{2}\right)^{\des(\sigma)}
		. \qquad (n\geq 1)\qedhere
\end{equation*}
\end{proof}

We could also provide a combinatorial proof of
Theorem \ref{th:EnComb} by using the {\it modified Foata-Strehl 
action} \cite{Tirrell2019Zhuang, Andre1879, Foata1974Strehl, Foata1976Strehl, Branden2008, Lin2015Zeng, Shapiro1983etal, Zhuang2017}. 
Actually, the combinatorial proof of identity \eqref{eq:W1} can be found in
\cite{Tirrell2019Zhuang}.
Notice that the two weight functions $W_3$ and $W_4$ are connected by 
	Stembridge's formula \cite{Stembridge1997, Zhuang2017, Branden2008}.


\section{Proofs of Theorems \ref{th:En_H} and \ref{th:En_Det}} \label{sec:proofs}
In his work on combinatorial aspects of continued fractions Flajolet \cite[Theorem 3A]{Flajolet1980}
obtained the continued fraction for the ordinary generating
function of the quadruple statistic 
$(\val, \pk, \da, \dd)$.
\footnote{\, There is a typo in \cite[Theorem 3A]{Flajolet1980}, the first numerator in the continued fraction of $P(u,v,w,z)$ should be $z$, instead of $1$.
This typo had not been fixed in the reprint \cite{Flajolet2006} of the paper.
}
By adding a superfluous variable $u_1$, 
because $\val(\sigma)=\pk(\sigma)+1$ for each permutation $\sigma$,
we restate his theorem as follows.

\begin{Theorem}\label{th:Flajolet3A}
	We have
	\begin{multline}\label{eq:Flajolet3A}
		\sum_{n\geq 1}  {x^{n}} \sum_{\sigma \in \mathfrak{S}_n}
	 u_1^{\val(\sigma)}u_2^{\pk(\sigma)} u_3^{\da(\sigma)} u_4^{\dd(\sigma)}\\
	=
\cFrac{u_1x}{1-1(u_3+u_4)x} - \cFrac{1\cdot 2u_1u_2x^2}{1-2(u_3+u_4)x} 
	-\cFrac{2\cdot 3 u_1u_2x^2}{1-3(u_3+u_4)x} - \cdots
	\end{multline}
\end{Theorem}

At this stage
it is interesting to compare the previous continued fraction expression with the continued fraction derived for 
the {\it exponential} generating function,
as stated in Theorem \ref{th:Exp_Fraction}

From the definitions of $W_2(\sigma)$ and 
$P_n(t,s)$ given in \eqref{eq:defW2} and \eqref{eq:Pn}, respectively,
we have $E_n=P_n(1/2, 0)$ by Theorem \ref{th:EnComb}.
Hence, 
the specialization  of identity \eqref{eq:Flajolet3A} 
with $u_1=1, u_2=1/2, u_3=0, u_4=1$ leads the following Theorem.
\footnote{\,Alan Sokal has informed me that continued fraction \eqref{eq:Enxn_1pFraction}
is known. 
As written in \cite{Sokal2019EM}, Jiang Zeng showed that 
\eqref{eq:Enxn_1pFraction}
is a consequence of a formula obtained by Stieltjes \cite{Stieltjes1890}.
This continued fraction was also independently rediscovered by
Matthieu Josuat-Verg\`es \cite{JosuatVerges2014} and
Alan Sokal \cite{Sokal2019EM}.
}
\begin{Theorem}\label{th:Enxn}
We have the following continued fraction for the Euler numbers
\begin{equation}\label{eq:Enxn_1pFraction}
		\sum_{n\geq 0} E_{n} x^n =
 1+ \cFrac{x}{1-x } - \cFrac{x^2}{1-2x } - \cFrac{3x^2}{1-3x } 
	- \cFrac{6x^2}{1-4x } 
	-\cdots
\end{equation}
The general pattern for the coefficients $a_k$ and $b_k$ are: 
\begin{equation*}
a_1=x,\quad      	a_k =-\binom{k}{2}x^2;           \qquad 
b_0=1,\quad 	b_k = 1-kx .
\end{equation*}
\end{Theorem}

\medskip
Notice that the above continued fraction is neither a super $1$-fraction
nor a super $2$-fraction.
Fortunately, we can use it to derive a super $1$-fraction by a
series of chop contractions at appropriate positions.

\begin{Theorem} \label{th:En_super1}
	We have the following super $1$-fraction expansion:
	\begin{align}
	\sum E_n x^n=
& \ \cFrac{ 1}{ 1} 
- \cFrac{ x}{ 1}
- \cFrac{ x^2}{ 1-2x}
- \cFrac{ 3x^2}{ 1}
- \cFrac{ 3x}{ 1}
- \cFrac{ 2x}{ 1}
		- \cFrac{ 2x}{ 1}\label{eq:En_super1}\\
&\quad
- \cFrac{ 5x}{ 1} 
- \cFrac{ 15x^2}{ 1-6x}
- \cFrac{ 21x^2}{ 1}
- \cFrac{ 7x}{ 1}
- \cFrac{ 4x}{ 1}
- \cFrac{ 4x}{ 1}
		- {\cdots} \nonumber
\end{align}
The general patterns for the coefficients $a_j$ and $b_j$ are:
\begin{align*}
	a_{1}&=1                         ,&  b_{0}&=0               ,       \\ 
	a_{6k} &=-2kx                  ,&  b_{6k} &= 1          ,       \\ 
	a_{6k+1} &=-2kx                  ,&  b_{6k+1} &= 1          ,       \\ 
	a_{6k+2} &=-(4k + 1)x            ,&  b_{6k+2} &= 1          ,      \\ 
	a_{6k+3} &=-(4k + 1)(2k + 1)x^2  ,&  b_{6k+3} &= 1-2(2k + 1)x ,    \\ 
	a_{6k+4} &=-(4k + 3)(2k + 1)x^2  ,&  b_{6k+4} &= 1              ,  \\ 
	a_{6k+5} &=-(4k + 3)x            ;&  b_{6k+5} &= 1              .  
\end{align*}

\end{Theorem}
\begin{proof} 
	The theorem will be proved by using the chop contraction
	defined in \eqref{eq:chop} at specific positions.
	Chop contraction at the  first position of the
	continued fraction on the right-hand side of \eqref{eq:En_super1}, we get
\begin{equation*}
{ 1} 
	+ \cFrac{ x}{ 1-x}
- \cFrac{ x^2}{ 1-2x}
- \cFrac{ 3x^2}{ 1}
- \cFrac{ 3x}{ 1}
- \cFrac{ 2x}{ 1}
- \cFrac{ 2x}{ 1}\\
- \cFrac{ 5x}{ 1} 
- \cFrac{ 15x^2}{ 1-6x}
 - {\cdots}
\end{equation*}
Then, chop at the 4th position:
\begin{equation*}
{ 1} 
	+ \cFrac{ x}{ 1-x}
- \cFrac{ x^2}{ 1-2x}
- \cFrac{ 3x^2}{ 1-3x}
- \cFrac{ 6x^2}{ 1-2x}
- \cFrac{ 2x}{ 1}\\
- \cFrac{ 5x}{ 1} 
- \cFrac{ 15x^2}{ 1-6x}
 - {\cdots}
\end{equation*}
and chop at the 5th position:
\begin{equation*}
{1} 
	+ \cFrac{ x}{ 1-x}
- \cFrac{ x^2}{ 1-2x}
- \cFrac{ 3x^2}{ 1-3x}
- \cFrac{ 6x^2}{ 1-4x}
- \cFrac{ 10x^2}{ 1-5x}\\
- \cFrac{ 15x^2}{ 1-6x}
 - {\cdots}
\end{equation*}
The next chop contraction is to be applied at position 8.
In general,
we contract the second numerator that is a monomial in $x$ of degree $1$, and repeat.
This will work since
\begin{equation*}
\begin{split}
	&-  \cFrac{ (4k + 3)x  }{ 1}
- \cFrac{ (2k+2)x}{ 1}
- \cFrac{ (2k+2)x}{ 1}
+ {\cdots} \\
	= &
	-   (4k + 3)x  
	- \cFrac{ (4k+3)(2k+2)x^2}{ 1-(2k+2)x}
- \cFrac{ (2k+2)x}{ 1}
+ {\cdots}
\end{split}
\end{equation*}
and
\begin{equation*}
	\begin{split}
		&-  \cFrac{ (2k + 2)x  }{ 1}
- \cFrac{ (4k+5)x}{ 1}
	- \cFrac{ (4k+5)(2k+3)x^2}{ 1-2(2k+3)x}
+ {\cdots} \\
		=&
-   (2k + 2)x 
	- \cFrac{ (4k+5)(2k+2)x^2}{ 1-(4k+5)}
	- \cFrac{ (4k+5)(2k+3)x^2}{ 1-2(2k+3)x}
+ {\cdots}
	\end{split}
\end{equation*}
Finally, we will get the continued fraction on the right-hand side of \eqref{eq:Enxn_1pFraction}.
Hence, \eqref{eq:En_super1} is true.
We verify that the continued fraction is a super $1$-fraction, under the general super continued fraction form
\eqref{eq:super} with $\delta=1$ and  $(k_0, k_1, k_2, \ldots) = (0,0,1,0,0,0)^*$, where
the star sign means that the sequence is periodic and obtained by repeating the 
underlying segment.
\end{proof}

\begin{proof}[Proof of Theorem \ref{th:En_H}]
We prove the result by
applying the even contraction on the  super $1$-fraction \eqref{eq:En_super1} given in Theorem \ref{th:En_super1}. 
Let us detail only the calculations for the coefficients $a_j'$. 
	From the general formula for the even contraction \eqref{eq:EvenContraction} we have
\begin{align*}
	a_1' &=  a_1 b_2= 1,   \cr
	a_2' &=  - a_{2}a_{3} { b_4} = -x^3, \cr
	a_{3k}' &=  - a_{6k-2} a_{6k-1} {b_{6k-4}} {  b_{6k}}
	 = -(4k-1)^2(2k-1)x^3 ,\\
	a_{3k+1}' &=  - a_{6k} a_{6k+1} {b_{6k-2}} {  b_{6k+2}}
	= -4k^2x^2, \\
	a_{3k+2}' &=  - a_{6k+2} a_{6k+3} {b_{6k}} {  b_{6k+4}}
	=  - (4k+1)^2 (2k+1)x^3.
\end{align*}
The calculations for the coefficients $b_j'$ are similar.
	We verify that the continued fraction \eqref{eq:En_H} is a super $2$-fraction, under the general super continued fraction form
\eqref{eq:super} with $\delta=2$ and  $(k_0, k_1, k_2, \ldots) = (0,1,0)^*$.
\end{proof}

\begin{proof}[Proof of Theorem \ref{th:En_Det}]
The Hankel determinants are evaluated by using the fundamental
theorem \ref{th:super2}.
For the Hankel continued fraction given in \eqref{eq:En_H}, we have
$(k_0, k_1, k_2, \ldots)=(0,1,0)^*$. So that
$$(s_0, s_1, s_2, \ldots)= (0,1,3,4,5,7,8,9,11,12,13,15\ldots)$$ 
and
$$(\epsilon_0, \epsilon_1,\epsilon_2, \ldots) 
	= (0,0,1,1,1,2,2,2,3,3,3,4,\ldots). $$
	Comparing \eqref{eq:En_H} and \eqref{eq:super}, we have
\begin{align*}
	v_0&=1 , \\
	v_{3k} &=4k^2 , \\
	v_{3k+1} &=(4k + 1)^2(2k + 1) , \\
	v_{3k+2} &=(4k +3)^2(2k + 1).
\end{align*}
	Put all these $(v_j), (s_j), (\epsilon_j)$ into \eqref{eq:HankelDetFundamental},  we 
	obtain the explicit Hankel determinant formulas given in Theorem \ref{th:En_Det}
	after simplification.
\end{proof}


\section{The ordinary generating functions of the Euler numbers}\label{sec:ordgf} 
In this section we consider the ordinary generating functions of the Euler
numbers, and derive some Hankel continued fractions and some Hankel
determinants involving these numbers.
Some formulas are known or easy to prove. We list them here for a quick view
and comparison.  Let $(\sec(x))^r=\sum_{n\geq 0} E_{2n}^{(r)} \frac {x^{2n}}{(2n)!}$.

\begin{Theorem}\label{th:MainHFrac_ord}
	We have the following Hankel continued fraction expansions.
	\begin{equation}\tag{F1}
	\sum_{n\geq 0} E_{2n}^{(r)} x^{2n} = 
	\cFrac {1}{1} - \cFrac{1rx^2}{1}
	- \cFrac{2(r+1)x^2}{1}- \cFrac{3(r+2)x^2}{1} - \cdots
		\qquad
\end{equation}
\begin{equation}\tag{F2}
	\sum_{n\geq 0} E_{2n}^{(r)} x^{n} = 
\cFrac{1}{1-rx} - \cFrac{2r(r + 1)x^2}{1-(5r + 8)x } 
- \cFrac{-12(r+2)(r+3)x^2}{1-(9r + 32)x } 
-\cdots
	\end{equation}
The general patterns for the coefficients $a_j$ and $b_j$ are:
\begin{equation*}
	\begin{gathered}
	a_1=1,\quad  a_k =-(2k + r - 3)(2k + r - 4)(2k - 3)(2k - 2)x^2;\\
	b_0=0,\quad  b_k = 1-(8k^2 + 4kr - 16k - 3r + 8)x.
	\end{gathered}
\end{equation*}
\begin{equation}\tag{F3}
	\sum_{n\geq 1} E_{2n}^{(r)} x^{n-1} 
	=\cFrac{r}{1 -(2+3r)x} - \cFrac{6(r + 2)(r + 1)x^2}{1 -(18+7r)x} 
	+\cdots
	\qquad\qquad
\end{equation}
The general patterns for the coefficients $a_j$ and $b_j$ are:
	\begin{align*}
		\begin{gathered}
		a_{1}=rx,  \quad
	a_{k} =-2(2k + r - 2)(2k + r - 3)(2k - 1)(k - 1)x^2;  \\ 
		b_{0}=1 , \quad
		b_{k} =1-( 8k^2 + 4kr - 8k - r + 2 )x. 
		\end{gathered}
\end{align*}
\begin{equation}\tag{F4}
\sum_{n\geq 0} E_{2n+1} x^{2n} = 
	\cFrac {1}{1} - \cFrac{1\cdot 2x^2}{1}
	- \cFrac{2\cdot 3x^2}{1}- \cFrac{3\cdot 4x^2}{1} - \cdots
\end{equation}
\begin{equation}\tag{F5}
\sum_{n\geq 0} E_{2n+1} x^{n} = 
\cFrac{1}{-2x + 1} + \cFrac{-12x^2}{-18x + 1} + \cFrac{-240x^2}{-50x + 1} 
+ \cdots
\end{equation}
The general patterns for the coefficients $a_j$ and $b_j$ are:
	\begin{equation*}
		\begin{gathered}
		a_1=1, \quad 
		a_{k} =-4(2k - 1)(2k - 3)(k - 1)^2x^2; \\
		b_0=0, \quad 
		b_{1} = -2(2k - 1)^2x + 1.
		\end{gathered}
	\end{equation*}
\begin{equation}\tag{F6}
\sum_{n\geq 1} E_{2n+1} x^{n-1} = 
	\cFrac{2}{1-8x} - \cFrac{72x^2}{1-32x } - \cFrac{600x^2}{1-72x } 
	+\cdots
	\qquad
\end{equation}
The general patterns for the coefficients $a_j$ and $b_j$ are:
\begin{equation*}
	\begin{gathered}
	a_{1}=2 , \quad 
	a_{k} =-4(2k - 1)^2(k - 1)kx^2;  \qquad
		b_{0}=0 , \quad 
		b_{k} = 1-8k^2x. 
	\end{gathered}
\end{equation*}
\begin{equation}\tag{F7}
		\sum_{n\geq 0} E_{n+1} x^n =
 \cFrac{1}{1-x } - \cFrac{x^2}{1-2x } - \cFrac{3x^2}{1-3x } 
	- \cFrac{6x^2}{1-4x } 
	-\cdots
\end{equation}
The general patterns for the coefficients $a_k$ and $b_k$ are given by
\begin{equation*}
	a_1=1 ,\quad 
	a_k =-\binom{k}{2}x^2; \qquad 
	b_0=0 , \quad  
b_k = 1-kx. 
\end{equation*}
	\begin{multline}\tag{F8}
	\sum_{n\geq 1} \frac {E_{2n}^{(r)}}{r}  {x^{2n-1}}=
	\cFrac{ x}{ 1 -(3r +2)x^2} \\
	\qquad - \cFrac{ 2\cdot 3(r + 2)(r + 1) x^4  }{  1-(7r+18)x^2 }
	- \cFrac{ 4\cdot 5 (r + 4)(r + 3)  x^4}{ 1-(11r+50)x^2} 
	- \cdots 
\end{multline}
The general patterns for the coefficients $a_j$ and $b_j$ are:
\begin{equation*}
	\begin{gathered}
	a_{1}=x,\quad 
a_{k} = -(2k - 1)(2k-2)   (2k  - 3 +r)(2k - 2+r)x^4;   \\
	b_{0}=0 , \quad 
	b_{k}= 1  -(8k^2 - 8k  +4rk + 2-r)x^2.
	\end{gathered}
\end{equation*}
\begin{equation}\tag{F9}
	\sum_{n\geq 0} E_{2n+1}x^{2n+1} = \cFrac{ x}{ 1 -2\cdot 1^2 x^2} 
	- \cFrac{ 1\cdot 2^2\cdot 3 x^4  }{  1-2\cdot 3^2x^2 }
	- \cFrac{ 3\cdot 4^2\cdot 5   x^4}{ 1-2\cdot 5^2x^2} 
	- \cdots 
\end{equation}
The general patterns for the coefficients $a_j$ and $b_j$ are:
\begin{equation*}
	\begin{gathered}
	a_{1}=x, \quad 
a_{k} = -(2k - 1)(2k-2)^2   (2k  - 3 )x^4;   \\
	b_{0}=0, \quad 
	b_{k}= 1  -2(2k-1)^2x^2.
	\end{gathered}
\end{equation*}
\begin{multline}\tag{F10}
		\sum_{n\geq 0} E_{n+2} x^n 
		= \cFrac{1}{1-2x} - \cFrac{x^2}{1-4x} 
	- \cFrac{18x^3}{1 - 4x-16x^2 } - \cFrac{50x^3}{1-8x} \\
	     - \cFrac{9x^2}{1-10x } 
	- \cFrac{196x^3}{1 -8x -64x^2 } - \cFrac{324x^3}{1-14x}
	-\cdots
\end{multline}
The general patterns for the coefficients $a_j$ and $b_j$ are:
\begin{align*}
	a_{1}&=1,                    &  b_{0}&=0                   ,        \\
	a_{3k+0} &=-2(4k - 1)^2kx^3, &  b_{3k+0} &= -16k^2x^2 - 4kx + 1 ,   \\
	a_{3k+1} &=-2(4k + 1)^2kx^3, &  b_{3k+1} &= -2(3k + 1)x + 1     ,   \\
	a_{3k+2} &=-(2k + 1)^2x^2;   &  b_{3k+2} &= -2(3k + 2)x + 1 .       
\end{align*}
\begin{equation}\tag{F11}
\sum_{n\geq 1} E_{2n+1} x^{2n-1} = 
	\cFrac{2x}{1-8x^2} - \cFrac{72x^4}{1-32x^2 } - \cFrac{600x^4}{1-72x^2 } 
	+\cdots
	\qquad
\end{equation}
The general patterns for the coefficients $a_j$ and $b_j$ are:
\begin{equation*}
	\begin{gathered}
	a_{1}=2x , \quad 
	a_{k} =-4(2k - 1)^2(k - 1)kx^4;  \\ 
		b_{0}=0 , \quad 
		b_{k} = 1-8k^2x^2. 
	\end{gathered}
\end{equation*}
\end{Theorem}

\begin{proof}
	(F1) This is a well-known formula (see \cite[p. 206]{Wall1948}, \cite{Flajolet1980}).

	(F2) First, replace $x^2$ by $x$ in (F1) we get
	\begin{equation}\label{eq:F1z}
	\sum_{n\geq 0} E_{2n}^{(r)} z^{n} = 
	\cFrac {1}{1} - \cFrac{1rz}{1}
	- \cFrac{2(r+1)z}{1}- \cFrac{3(r+2)z}{1} - \cdots
\end{equation}
	Even contraction on \eqref{eq:F1z} yields (F2).

	(F3) Odd contraction on \eqref{eq:F1z}. Then, subtract by 1 and divide by $x$. 

	(F4) Divide $x$ in \eqref{eq:E2np1}. 

	(F5) First, replace $x^2$ by $x$  in (F4), we get
\begin{equation}\label{eq:F4x}
\sum_{n\geq 0} E_{2n+1} z^{n} = 
	\cFrac {1}{1} - \cFrac{1\cdot 2z}{1}
	- \cFrac{2\cdot 3z}{1}- \cFrac{3\cdot 4z}{1} - \cdots
\end{equation}
	Even contraction on \eqref{eq:F4x} yields (F5).

(F6) Odd contraction on \eqref{eq:F4x}. Then, subtract by 1 and divide by $x$. 

(F7) Subtract by 1 and divide by $x$ in \eqref{eq:Enxn_1pFraction}.

(F8)  Chop contraction on (F1) yields
\begin{equation*}
	\sum_{n\geq 0}E_{2n}^{(r)} x^{2n}=1 + \cFrac{rx^2}{1-rx^2 } 
	- \cFrac{2(r + 1)x^2}{1} - \cFrac{3(r + 2)x^2}{1} - \cFrac{4(r + 3)x^2}{1} -  \cdots
\end{equation*}
Then, subtract by 1 and divide by $rx$ in the above identity. We get
\begin{equation*}
	\sum_{n\geq 1} \frac{E_{2n}^{(r)}}{r} x^{2n-1}= \cFrac{x}{1-rx^2 } 
	- \cFrac{2(r + 1)x^2}{1} - \cFrac{3(r + 2)x^2}{1} - \cFrac{4(r + 3)x^2}{1} -  \cdots
\end{equation*}
with the general patterns
\begin{equation*}
	\begin{gathered}
a_1=x ,\quad 
a_k =-(k + r - 1)kx^2;  \\
	b_0=0,\quad 
	b1=1-rx^2  ,\quad 
b_k = 1. 
	\end{gathered}
\end{equation*}
	Notice that the previous continued fraction is {\it not} a super $1$-fraction. 
	Even contraction on the above fraction yields (F8), which is a 
	$H$-fraction.

	(F9) Take $w=0$ in (F8). We get (F9), by using the fact that 
$$\sum_{n\geq 1} \frac {E_{2n}^{(r)}}{r} \frac {x^{2n-1}}{(2n-1)!}
=	\frac{1}{r}(\sec(x)^r)' =\tan(x)\sec(x)^r.
$$

	(F10) This continued fraction is very similar to that given in  Theorem \ref{th:En_H}. However, unlike \eqref{eq:En_H}, which has a super $1$-fraction (see Theorem \ref{th:En_super1}), (F10) does not have a super $1$-fraction.
	Thanks to the similarity of (F10) and \eqref{eq:En_H}, this proof is suggested by the proof of Theorem \ref{th:En_H}.

	Apply the haircut contraction as defined in \eqref{eq:Haircut} to (F7) with $\alpha=1$. We get
\begin{equation*}		
	\sum_{n\geq 0} E_{n+1} x^n 
		=1+ \cFrac{x}{1-x } - \cFrac{x}{1-x } - \cFrac{3x^2}{1-3x } 
	- \cFrac{6x^2}{1-4x } -\cdots\\
\end{equation*}
Hence
\begin{equation}\label{eq:Enp2_xx}
	\sum_{n\geq 0} E_{n+2} x^n 
		= \cFrac{1}{1-x } - \cFrac{x}{1-x } - \cFrac{3x^2}{1-3x } 
	- \cFrac{6x^2}{1-4x } -\cdots\\
\end{equation}
This is neither a super $1$-fraction, nor a super $2$-fraction.
Now, we claim that
\begin{multline}\label{eq:Enp2_notsuper1}
	\sum_{n\geq 0} E_{n+2} x^n =
\;\cFrac{1}{1-x } - \cFrac{x}{1} - \cFrac{x}{1} - \cFrac{3x}{1} 
	 - \cFrac{6x^2}{1-4x }\\
	 - \cFrac{10x^2}{1} 
	 - \cFrac{5x}{1} 
	- \cFrac{3x}{1} - \cFrac{3x}{1} -\cdots 
\end{multline}
with the general patterns
\begin{align*}
	a_{1}&=1          ,&                    b_{0}=0, &\quad b_{1}=1-x ,      \\ 
	a_{6k+0} &=-2(4k + 1)kx^2 ,&            b_{6k+0} &= 1,                     \\ 
	a_{6k+1} &=-(4k + 1)x       ,&          b_{6k+1} &= 1,                     \\ 
	a_{6k+2} &=-(2k + 1)x       ,&          b_{6k+2} &= 1,                     \\ 
	a_{6k+3} &=-(2k + 1)x         ,&        b_{6k+3} &= 1,                     \\ 
	a_{6k+4} &=-(4k + 3)x          ,&       b_{6k+4} &= 1,                     \\ 
	a_{6k+5} &=-2(4k + 3)(k + 1)x^2  ;&     b_{6k+5} &= 1-4(k + 1)x.       
\end{align*}
Notice that \eqref{eq:Enp2_notsuper1} is neither a super $1$-fraction nor a super $2$-fraction.

	By using the chop contractions at appropriate positions in the
	right-hand side of
	\eqref{eq:Enp2_notsuper1}, we successively get
\begin{align*}
	&\cFrac{1}{1-x } - \cFrac{x}{1} - \cFrac{x}{1} - \cFrac{3x}{1} 
	- \cFrac{6x^2}{1-4x } - \cFrac{10x^2}{1} - \cFrac{5x}{1} 
	- \cFrac{3x}{1} - \cdots\\
& = \cFrac{1}{1-x } - \cFrac{x}{1-x} -  \cFrac{3x^2}{1-3x} 
	- \cFrac{6x^2}{1-4x } - \cFrac{10x^2}{1} - \cFrac{5x}{1} 
	- \cFrac{3x}{1} - \cdots\\
& = \cFrac{1}{1-x } - \cFrac{x}{1-x} -  \cFrac{3x^2}{1-3x} 
	- \cFrac{6x^2}{1-4x } - \cFrac{10x^2}{1-5x}  
	- \cFrac{15x^2}{1-3x} - \cdots
\end{align*}
By using the general patterns we can obtain the right-hand side of 
\eqref{eq:Enp2_xx},  so that \eqref{eq:Enp2_notsuper1} is true.
Finally, an even contraction on \eqref{eq:Enp2_notsuper1} yields (F10).
We verify that (F10) is a $H$-fraction.

(F11) Replace $x$ by $x^2$ and multiply by $x$ in (F6).
\end{proof}

By using Theorem \ref{th:super2}, the Hankel continued fractions
(F1--F11) listed in Theorem \ref{th:MainHFrac_ord} implies the 
Hankel determinants formulas (H1--H11) in the next theorem respectively.

\begin{Theorem}\label{th:MainHDet_ord}
	We have the following formulas for the Hankel determinants.

\text{\rm (H1)}
The Hankel  determinants of $(E_0^{(r)}, 0, E_2^{(r)}, 0, E_4^{(r)}, \ldots)$ are
$$
H_n = \prod_{k=1}^{n-1} k! r(r+1)(r+2)\cdots (r+k-1).
$$
In particular, when $r=1$,
the Hankel determinants of $(E_0, 0, E_2, 0, E_4, \ldots)$ are
$$
H_n = \prod_{k=1}^{n-1} k!^2. 
$$

\text{\rm	(H2)}
The Hankel determinants of $(E_0^{(r)}, E_2^{(r)}, E_4^{(r)}, \ldots)$ are
$$
H_n= \prod_{k=1}^{n-1} (2k)! r(r+1)\cdots (r+2k-1).
$$
In particular, when $r=1$, 
the Hankel determinants of $(E_0, E_2, E_4, \ldots)$ are
$$
H_n= \prod_{k=1}^{n-1} (2k)!^2. 
$$

\text{\rm (H3)}
The Hankel determinants of $(E_2^{(r)}, E_4^{(r)}, \ldots)$ are
$$
	H_n= \prod_{k=0}^{n-1} (2k+1)! r(r+1)\cdots (r+2k).
$$
In particular, when $r=1$, 
the Hankel determinants of $(E_2, E_4, \ldots)$ are
$$
H_n= \prod_{k=1}^{n-1} (2k+1)!^2 
$$

	\text{\rm	(H4)}
The Hankel determinants of
$(E_1, 0, E_3, 0, E_5, \ldots)$ are
$$
H_n = n!\prod_{k=1}^{n-1} k!^2.
$$

	\text{\rm	(H5)}
The Hankel determinant of
$(E_1,  E_3,  E_5, \ldots)$ are
$$
H_n = \prod_{k=1}^{2n-1} k!.
$$

	\text{\rm	(H6)}
The Hankel determinant of
$(E_3,  E_5, \ldots)$ are
$$
H_n = \prod_{k=1}^{2n} k!.
$$

	\text{\rm(H7)}
	The Hankel determinants of $(E_1, E_2, E_3, E_4, \ldots)$ are
$$
	H_n= \frac{n!}{2^{n(n-1)/2}} \prod_{k=2}^{n-1} k!^2
$$

	\text{\rm (H8)}
The Hankel determinants of  $(0, E_2^{(r)}/r, 0, E_4^{(r)}/r, \ldots )$ are
$$
H_{2n+1}=0, \quad
H_{2n}= (-1)^n \prod_{k=1}^{n-1} \left((2k+1)! (r+1)(r+2)\cdots (r+2k)\right)^2.
$$
Or equivalently, the Hankel determinants of  $(0, E_2^{(r)}, 0, E_4^{(r)}, \ldots )$ are
$$
H_{2n+1}=0, \quad
H_{2n}= (-1)^n r^2\prod_{k=1}^{n-1} \left((2k+1)! r(r+1)(r+2)\cdots (r+2k)\right)^2.
$$
In particular, when $r=1$, the Hankel determinants of  $(0, E_2, 0, E_4, \ldots )$ are
$$
H_{2n+1}=0, \quad
H_{2n}= (-1)^n \prod_{k=1}^{n-1} (2k+1)!^4.
$$

\text{\rm (H9)}
The Hankel determinants of  $(0, E_1, 0, E_3,0,E_5, \ldots )$ are
$$
H_{2n+1}=0, \quad
H_{2n}= (-1)^n \prod_{k=1}^{2n-1} k!^2.
$$

\text{\rm (H10)}
The Hankel determinants  of
	$(E_2,E_3, E_4,E_5, \ldots)$ 
	 are
\begin{align*}
	H_0&=1,\\
	H_{4k}&= \frac{ (-1)^{k}k^2(2k-1)!^2 }{ 2^{8k^2-4k-2} }
		\prod_{j=1}^{2k-1}(2j+1)!^4  \\
	H_{4k+1}&= \frac{ (-1)^{k}(2k)!^2 }{ 2^{8k^2} (4k+1)!^2 }
		\prod_{j=1}^{2k}(2j+1)!^4  \\
	H_{4k+2}&= \frac{ (-1)^{k}(2k+1)!^2 }{ 2^{8k^2+4k} }
		\prod_{j=1}^{2k}(2j+1)!^4  \\
		H_{4k+3}&=0
\end{align*}

	\text{\rm	(H11)}
The Hankel determinant of
$(0, E_3, 0, E_5, \ldots)$ are
$$
H_n = (-1)^n \prod_{k=1}^{2n} k!^2.
$$

\end{Theorem}



\section{The exponential generating functions of the Euler numbers}\label{sec:expgf} 
In this section we consider the exponential generating functions of the Euler numbers. They are $\tan(x), \sec(x), \tan(x)+\sec(x)$ and their variants.  
Although most continued fractions and Hankel determinants involving 
the Euler numbers are for the {\it ordinary} generating functions (see Section \ref{sec:ordgf}), a few of them are about the 
{\it exponential} generating functions. 
In 1761, Lambert \cite{Lambert1761} proved that $\pi$ is irrational by first 
deriving the following continued fraction expansion of 
$\tan(x)$
(see \cite[p. 349 (91.7)]{Wall1948})
\begin{equation}\label{eq:tan_Lambert}
	\tan(x)= 
	\cFrac{x}{1} - \cFrac {x^2}{3} -\cFrac {x^2}{5} -\cFrac {x^2}{7}  -\cdots
\end{equation}
or \cite[p. 349 (91.6)]{Wall1948} \footnote{\,There are two typos in the equalities (91.6). 
The middle side
$\frac{z\Psi(\frac 32; \frac z4)^2}{\Psi(\frac 32; \frac z4)^2} $
should be
$\frac{z\Psi(\frac 32; \frac {z^2}{4})}{\Psi(\frac 32; \frac {z^2}{4})}$.}
\begin{equation}\label{eq:tanh_Lambert}
	\tanh(x)= 
	\cFrac{x}{1} + \cFrac {x^2}{3} +\cFrac {x^2}{5} +\cFrac {x^2}{7}  +\cdots
\end{equation}
Also, Hankel determinants of the Euler numbers divided by the factorial numbers $E_n/n!$ are studied in \cite{Morales2017PakII}.

We have seen the Flajolet continued fraction \eqref{eq:Flajolet3A} for the ordinary generating function
of the  quadruple statistic $(\val, \pk, \da, \dd)$ \cite{Flajolet1980},
and the Carlitz-Scoville exponential generating function \eqref{eq:CarlitzScoville}  for the same statistic \cite{Carlitz1974Scoville}. The next
continued fraction for their exponential generating function seems to be new.
\begin{Theorem}\label{th:Exp_Fraction}
We have the following
 continued fraction of the exponential generating function
for the quadruple statistic $(\val, \pk, \da, \dd)$: 
	\begin{multline}\label{eq:Exp_Fraction}
		\sum_{n\geq 1} \frac {x^{n}}{n!} \sum_{\sigma \in \mathfrak{S}_{n}}
	u_1^{\val(\sigma)} u_2^{\pk(\sigma)} u_3^{\da(\sigma)} u_4^{\dd(\sigma)} \\
		=\cFrac{u_1 x}{1-cx} + \cFrac{{(c^2-u_1u_2)}x^2}{3} 
		+\cFrac{{(c^2-u_1u_2)} x^2}{5} 
		+\cFrac{{(c^2-u_1u_2)} x^2}{7} 
		+ \cdots
	\end{multline}
	where $c=(u_3+u_4)/2$.
\end{Theorem}
\begin{proof}
	Let $F(x)$ be the left-hand side of \eqref{eq:Exp_Fraction}.
By using Carlitz-Scoville formula \eqref{eq:CarlitzScoville},
with $u_3+u_4=\alpha_1+\alpha_2$, $u_1u_2=\alpha_1 \alpha_2$,
	and $\tau=(\alpha_1-\alpha_2)/2$,
we have
	\begin{equation*}
\begin{split}
	F(x)
	&=\frac {u_1}{x} \cdot \frac {e^{\alpha_2 x} - e^{\alpha_1 x}}{\alpha_2 e^{\alpha_1 x} - \alpha_1 e^{\alpha_2 x}},\\
	&=\frac{u_1}{x}\cdot \frac {e^{\tau x} - e^{-\tau x}}{
		\tau (e^{\tau x} + e^{-\tau x}) - c ( e^{\tau x} - e^{-\tau x}) },\\
	&=\frac {u_1}{
		\frac{\tau x}{\tanh(\tau x)} - c x }.
\end{split}
\end{equation*}
By using Lambert's continued fraction \eqref{eq:tan_Lambert}, we obtain
	\begin{equation}\label{eq:FourStatsExp_ContFrac}
	F(x)=\cFrac{u_1}{1-cx}
	+\cFrac{\tau^2x^2}{3}
	+\cFrac{\tau^2x^2}{5}
	+\cFrac{\tau^2x^2}{7} + \cdots
	\end{equation}
	Replace $\tau^2$ by $(\alpha_1-\alpha_2)^2/2^2 = c^2-u_1u_2$,
	we obtain \eqref{eq:Exp_Fraction}.
\end{proof}

Comparing  \eqref{eq:Exp_Fraction} and \eqref{eq:Flajolet3A},
we can roughly say that the {\it formal Laplace transformation} converts the continued
fraction on the right-hand side of \eqref{eq:Exp_Fraction} to 
the continued fraction on the right-hand side of \eqref{eq:Flajolet3A}.

\begin{Theorem}\label{th:MainHFrac_exp}
	We have the following Hankel fraction expansions.

\begin{equation}\tag{F12}
	\sum_{n\geq 0} E_{2n+1}\frac{x^{2n}}{(2n+1)!}= 
	\cFrac{1}{1} - \cFrac{\frac{1}{1\cdot 3}x^2}{1} 
	- \cFrac{\frac{1}{3\cdot 5}x^2}{1} - \cFrac{\frac{1}{5\cdot 7}x^2}{1} -\cdots
\end{equation}

\begin{equation}\tag{F13}
	\sum_{n\geq 0} E_{2n+1}\frac{x^{n}}{(2n+1)!}
	=\cFrac{1}{1-\frac{1}{3}x} 
	- \cFrac{\frac{1}{1\cdot 3^2\cdot 5}x^2}{1-\frac{2}{3\cdot 7}x} 
	- \cFrac{\frac{1}{5\cdot 7^2 \cdot 9}x^2}{1-\frac{2}{7\cdot 11}x} 
+ \cdots
	\end{equation}

\begin{equation}\tag{F14}
	\sum_{n\geq 0} E_{2n+3}\frac{x^{2n+1}}{(2n+3)!}
	=\cFrac{\frac 13 x}{1-\frac{2}{1\cdot 5}x^2} 
	- \cFrac{\frac{1}{3\cdot 5^2\cdot 7}x^4}{1-\frac{2}{5\cdot 9}x^2} 
	- \cFrac{\frac{1}{7\cdot 9^2 \cdot 11}x^4}{1-\frac{2}{9\cdot 11}x^2} 
	- \cdots
\end{equation}

\begin{equation}\tag{F15}
	\tan(x)=\cFrac{x}{1-\frac{1}{3}x^2} 
	- \cFrac{\frac{1}{1\cdot 3^2\cdot 5}x^4}{1-\frac{2}{3\cdot 7}x^2} 
	- \cFrac{\frac{1}{5\cdot 7^2 \cdot 9}x^4}{1-\frac{2}{7\cdot 11}x^2} 
	- \cFrac{\frac{1}{9 \cdot 11^2\cdot 13}x^4}{1-\frac{2}{11\cdot 15}x^2 } 
	- \cdots
\end{equation}

	\begin{equation}\tag{F16}
	\sum_{n\geq 0}{E_{n+1}\frac{x^{n}}{(n+1)!}} 
= \cFrac{1}{1-\frac{1}{2}x} 
- \cFrac{\frac{1}{2\cdot 6}x^2}{1} 
- \cFrac{\frac{1}{6\cdot 10}x^2}{1} 
- \cFrac{\frac{1}{10\cdot 14}x^2}{1} 
-\cdots
	\end{equation}

\begin{equation}\tag{F17}
\tan(x)+\sec(x)=
\cFrac{1}{1-x } 
+ \cFrac{\frac{1}{2}x^2}{1+\frac{2}{3}x } 
+ \cFrac{\frac{1}{6^2}x^2}{1-\frac{4}{3\cdot 5}x } 
+ \cFrac{\frac{1}{10^2}x^2}{1+\frac{6}{5\cdot 7}x} 
+ \cdots
\end{equation}

	\begin{equation}\tag{F18}
	\sum_{n\geq 0}{E_{n+2}\frac{x^{n}}{(n+2)!}} 
	=\cFrac{\frac{1}{2}}{1-\frac{2}{3}x } 
+ \cFrac{\frac{1}{6^2}x^2}{1+\frac{4}{3\cdot 5}x } 
+ \cFrac{\frac{1}{10^2}x^2}{1-\frac{6}{5\cdot 7}x} 
+ \cFrac{\frac{1}{14^2}x^2}{1+\frac{8}{7\cdot 9}x} 
+ \cdots
\end{equation}
\begin{equation}\tag{F19}
		\sum_{n\geq 0}{E_{n+3}\frac{x^{n}}{(n+3)!}}
		= \cFrac{ \frac {1}{3}}{1-\frac{5}{8}x} 
			- \cFrac{\frac{3}{320}x^2}{1+\frac{13}{72}x} 
		-\cFrac{\frac{16}{2835}x}{1-\frac{25}{288}x}  - \cdots\qquad\qquad
	\end{equation}
The general patterns of the coefficients $a_j$ and $b_j$ are:
\begin{equation*}
	a_1= \frac {1}{3},\
	a_{j} = -\frac{(j-1)^2(j+1)^2x^2}{4j^4(2j-1)(2j+1)};\
	b_0=0,\
	b_{j} = 1+ (-1)^j \frac{(2j^2+2j+1) x}{2j^2(j+1)^2}.
\end{equation*}
	\begin{equation}\tag{F20}
		\sum_{n\geq 0}{E_{n+4}\frac{x^{n}}{(n+4)!}}
		=
		 \cFrac{\frac{5}{24}}{1 - \frac{16}{25}x} 
		 + \cFrac{\frac{11}{3750}x^2}{1 + \frac{432}{1925}x} 
		 + \cFrac{\frac{475}{142296}x^2}{1 - \frac{640}{4389}x} 
		 + \cdots\qquad
\end{equation}
The general patterns for the coefficients $a_j$ and $b_j$ are:
\begin{equation*}
\begin{gathered}
a_1=\frac {5}{24}, \quad
a_j=\frac{{\left(j^{2} + 3 \, j + 1\right)} {\left(j^{2} - j - 1\right)} {\left(j + 2\right)} {\left(j - 1\right)} \cdot x^2}{4 \, {\left(j^{2} + j - 1\right)}^{2} {\left(2 \, j + 1\right)}^{2} {\left(j + 1\right)} j} ;
\\
b_0=\frac{1}{3},\quad
b_j=1+\frac{2 \, \left(-1\right)^{j} {\left(j + 2\right)} {\left(j + 1\right)}^{3} j\cdot x}{{\left(j^{2} + 3 \, j + 1\right)} {\left(j^{2} + j - 1\right)} {\left(2 \, j + 3\right)} {\left(2 \, j + 1\right)}}.
\end{gathered}
\end{equation*}
\begin{equation}\tag{F21}
	\sum_{n\geq 0} E_{2n+3}\frac{x^{2n}}{(2n+3)!}
	=\cFrac{\frac 13 }{1} 
	- \cFrac{\frac{2}{5}x^2}{1} 
	- \cFrac{\frac{1}{210}x^2}{1} 
	- \cFrac{\frac{5}{126}x^2}{1} 
	- \cdots
	\qquad
\end{equation}
The general patterns for the coefficients $a_j$  are:
\begin{equation*}
\begin{gathered}
	a_{2k} = \frac{-(2k+1)(k+1) }{(4k-1)(4k+1)(2k-1)k}, \\
	a1=\frac 13; \quad a_{2k+1} = \frac{-k(2k-1)}{(4k+1)(4k+3)(2k+1)(k+1)}.
\end{gathered}
\end{equation*}
\begin{equation}\tag{F22}
	\sum_{n\geq 0} E_{2n+3}\frac{x^{n}}{(2n+3)!}
	=\cFrac{\frac 13 }{1-\frac{2}{1\cdot 5}x} 
	- \cFrac{\frac{1}{3\cdot 5^2\cdot 7}x^2}{1-\frac{2}{5\cdot 9}x} 
	- \cFrac{\frac{1}{7\cdot 9^2 \cdot 11}x^2}{1-\frac{2}{9\cdot 11}x} 
	- \cdots
\end{equation}
\begin{equation}\tag{F23}
	\sum_{n\geq 0} E_{2n+5} \frac{x^n}{(2n+5)!}=
	\cFrac{ \frac {2}{15}}{1 - \frac {17}{42}x}
	-	\cFrac{ \frac {1}{5292}x^2}{1 - \frac {101}{2310}x}
	-	\cFrac{ \frac {56}{1061775}x^2}{1 - \frac {73}{4620}x}
- \cdots
\end{equation}
The general patterns for the coefficients $a_j$ and $b_j$ are:
\begin{equation*}
\begin{gathered}
a_1=\frac {2}{15}, \quad
a_j=
-\frac{{\left(2 \, j + 1\right)} {\left(2 \, j - 3\right)} {\left(j + 1\right)} {\left(j - 1\right)}x^2}{{\left(4 \, j + 1\right)} {\left(4 \, j - 1\right)}^{2} {\left(4 \, j - 3\right)} {\left(2 \, j - 1\right)}^{2} j^{2}};
\\
b_0=0,  \quad
b_j=1-
\frac{(8 \, j^{4} + 8 \, j^{3} + 22 \, j^{2} + 10 \, j + 3)x}{{\left(4 \, j + 3\right)} {\left(4 \, j - 1\right)} {\left(2 \, j + 1\right)} {\left(2 \, j - 1\right)} {\left(j + 1\right)} j}.
\end{gathered}
\end{equation*}
\begin{equation}\tag{F24}
	\sum_{n\geq 0} E_{2n+7} \frac{x^n}{(2n+7)!}=
	\cFrac{ \frac {17}{315} }{1 - \frac{62}{153}x}
	-	\cFrac{ \frac {26}{1287495}x^2}{1 - \frac {1150}{25857}x}
- \cdots
\qquad\qquad
\end{equation}
The general patterns for the coefficients $a_j$ and $b_j$ are:
\begin{equation*}
	\begin{gathered}
a_j=
\frac{-{\left(4 \, j^{2} + 10 \, j + 3\right)} {\left(4 \, j^{2} - 6 \, j - 1\right)} {\left(2 \, j + 3\right)} {\left(2 \, j - 3\right)} {\left(j + 2\right)} {\left(j - 1\right)}x^2}{{\left(4 \, j^{2} + 2 \, j - 3\right)}^{2} {\left(4 \, j + 3\right)} {\left(4 \, j + 1\right)}^{2} {\left(4 \, j - 1\right)} {\left(2 \, j + 1\right)} {\left(2 \, j - 1\right)} {\left(j + 1\right)} j};
\\
a_1=\frac {17}{315}; \quad
b_0=0,  \
b_j=1-
\frac{2 \, {\left(16 \, j^{4} + 48 \, j^{3} + 164 \, j^{2} + 192 \, j + 45\right)}x}{{\left(4 \, j^{2} + 10 \, j + 3\right)} {\left(4 \, j^{2} + 2 \, j - 3\right)} {\left(4 \, j + 5\right)} {\left(4 \, j + 1\right)}}.
	\end{gathered}
\end{equation*}
\end{Theorem}

\begin{proof}
(F12) Apply the equivalence transformations 
\cite[p. 19]{Wall1948}
on the Lambert continued fraction \eqref{eq:tan_Lambert} and divide by $x$.

(F13)
Replace $x^2$ by $x$ in (F12) we get
\begin{equation}\tag{f13}
	\sum E_{2n+1}\frac{x^{n}}{(2n+1)!}= 
	\cFrac{1}{1} - \cFrac{\frac{1}{1\cdot 3}x}{1} 
	- \cFrac{\frac{1}{3\cdot 5}x}{1} - \cFrac{\frac{1}{5\cdot 7}x}{1} -\cdots
\end{equation}
	Then, even contraction on (f13) yields (F13).

	(F14) Odd contraction on (F12); subtract by 1; divide by $x$.

	(F15) Replace $x$ by $x^2$ in (F13), and multiply by $x$. We get the $H$-fraction (F15).

(F16) 
By Theorem \ref{th:EnComb} with the weight  $W_2$, we have 
$E_n=P_n(1/2, 0)$. Hence, with the specialization $u_2=1/2, u_3=0, u_1=u_4=1$
	in Theorem \ref{th:Exp_Fraction},
	we have $c=1/2, c^2-u_1u_2= 1/4 - 1/2  = -1/4$.
Identity \eqref{eq:Exp_Fraction} becomes.
	\begin{equation*}
	\sum_{n\geq 1}{E_{n}\frac{x^{n}}{n!}} 
= \cFrac{x}{1-\frac{1}{2}x} 
- \cFrac{\frac{1}{4}x^2}{3} 
- \cFrac{\frac{1}{4}x^2}{5} 
- \cFrac{\frac{1}{4}x^2}{7} 
-\cdots
	\end{equation*}
Divide the above continued fraction by $x$, and 
normalize to $J$-fraction by 
equivalence transformations \cite[p. 19]{Wall1948}, 
we get (F16).

(F17)	We claim that
	\begin{equation}\tag{f17}
	\tan(x)+\sec(x)= 
\cFrac{1}{1} - \cFrac{x}{1} + \cFrac{\frac{1}{2}x}{1} + \cFrac{\frac{1}{6}x}{1} 
	- \cFrac{\frac{1}{6}x}{1} - \cFrac{\frac{1}{10}x}{1} 
  + \cdots
\end{equation}
The general patterns for the coefficients $a_j$ are:
\begin{equation*}
\begin{gathered}
a_1=1, \quad a_2=-x,
\\
a_{4k} = \frac{x}{8k-2},\quad
a_{4k+1} = -\frac{x}{8k-2},\quad
a_{4k+2} = -\frac{x}{8k+2},\quad
a_{4k+3} = \frac{x}{8k+2}.
\end{gathered}
\end{equation*}
	Odd contraction on (f17), subtract by $1$, divide by $x$, we get (F16). Hence, (f17) is true.
Even contraction on (f17) implies (F17).

(F18)
	Let $F(x)=\tan(x)+\sec(x)$ and $G(x)=((F(x)-1)/x-1)/x$.
	By (F17), it suffices to verify that
	$$
	F(-x)=\frac{1}{1+x+x^2G(x)}.
	$$
	This is true, since $F(x)F(-x)=1$.

(F19)
We have
	\begin{equation}\tag{f18}
		\sum_{n\geq 0}{E_{n+2}\frac{x^{n}}{(n+2)!}}
		=\cFrac{\frac{1}{2}}{1} - \cFrac{ \frac{2}{3} x}{1} +\cFrac{\frac{1}{24}x}{1} 
+ \cFrac{\frac{9}{40}x}{1}
		- \cdots
	\end{equation}
The	general patterns for the coefficients $a_j$ are:
\begin{equation*}
	a_1=1/2,\quad
	a_{2j} = \frac{(-1)^{j} (j+1)^2x}{2j^2(2j+1)},\quad
	a_{2j+1} = \frac{(-1)^{j+1} j^2x}{2(j+1)^2(2j+1)},
\end{equation*}
since, even contraction on (f18) implies (F18).
Now, odd contraction on (f18), subtract by $1/2$, divide by $x$, we get (F19).

	(F20)
	We have
	\begin{equation}\tag{f19}
		\sum_{n\geq 0}{E_{n+3}\frac{x^{n}}{(n+3)!}}
		=\cFrac{\frac{1}{3}}{1} - \cFrac{\frac{5}{8} x}{1} -\cFrac{\frac{3}{200}x}{1} + \cFrac{\frac{44}{225}x}{1}
		+ \cdots
	\end{equation}
The	general patterns for the coefficients $a_j$ are:
\begin{equation*}
\begin{gathered}
	a_{2j} = \frac{(-1)^{j} j(j+2)(j^2+3j+1)x}{2(j^2+j-1)(2j+1)(j+1)^2},\\
	a_1=1/3, \qquad 
	a_{2j+1} = \frac{(-1)^j j(j+2)(j^2+j-1)x}{2(j^2+3j+1)(2j+3)(j+1)^2},
\end{gathered}
\end{equation*}
since, even contraction on (f19) implies (F19).
Now odd contraction on (f19), subtract by $1/3$, divide by $x$, we get (F20).

(F21) Even contraction on (F21) is the same as (F14) divided by $x$. 

(F22) From (F14), divide by $x$, replace $x^2$ by $x$.

(F23) Replace $x^2$ by $x$ in (F21), we get the $S$-fraction:
\begin{equation}\tag{f22}
	\sum_{n\geq 0} E_{2n+3} \frac{x^n}{(2n+3)!}=
	\cFrac{ \frac {1}{3}}{1}
-	\cFrac{ \frac {2}{5}x}{1}
-	\cFrac{ \frac {1}{210}x}{1}
-	\cFrac{ \frac {5}{126}x}{1}
- \cdots
\end{equation}
For the general patterns see (F21).
Odd contraction on (f22), subtract by $1/3$, 
divide by $x$. We get (F23).

(F24) We claim the $S$-fraction:
\begin{equation}\tag{f23}
	\sum_{n\geq 0} E_{2n+5} \frac{x^n}{(2n+5)!}=
	\cFrac{ \frac {2}{15}}{1}
-	\cFrac{ \frac {17}{42}x}{1}
-	\cFrac{ \frac {1}{2142}x}{1}
-	\cFrac{ \frac {364}{8415}x}{1}
- \cdots
\end{equation}
The general patterns for the coefficients $a_j$ are:
\begin{equation*}
\begin{gathered}
	a_{2j} = \frac {- (4j+6)(4j+8)(4j^2+10j+3) x}{ (4j+1)(4j+2)(4j+3)(4j+4) 
(4j^2+2j-3)},\\
	a_1=\frac {2}{15},\qquad
	a_{2j+1} = \frac{-(4j-2)(4j)(4j^2+2j-3)   x}{(4j+2)(4j+3)(4j+4)(4j+5) (4j^2+10j+3   )}.
\end{gathered}
\end{equation*}
Even contraction on (f23), we get (F23). So that (f23) is true.
Now, odd contraction on (f23), subtract by $2/15$, 
divide by $x$. We get (F24).
\end{proof}

By using Theorem \ref{th:super2}, the Hankel continued fractions
(F12--F24) listed in Theorem \ref{th:MainHFrac_exp} imply the 
Hankel determinants formulas (H12--H24) in the next theorem respectively.

\begin{Theorem}\label{thMainHDet_exp}
	We have the following formulas for the Hankel determinants.

	\text{\rm (H12)}
The Hankel determinants of $(E_1/1!, 0 , E_3/3!, 0, E_5/5!, 0, \ldots)$ are
$$
H_0=1, \quad
H_n= 2^{(n-1)^2} \frac {(n-1)!}{(2n-1)!} 
\prod_{k=1}^{n-1} \frac {(k-1)!^2}{(2k-1)!^2}.
$$

	\text{\rm (H13)}
The Hankel determinants of $(E_1/1!,  E_3/3!, E_5/5!,  \ldots)$ are
$$
H_0=1, \quad H_n= 2^{(n-1)(2n-1)} \prod_{k=1}^{2n-2} \frac {k!}{(2k+1)!}.
$$

\text{\rm (H14)}
The Hankel determinants of $(0,  E_3/3!,0, E_5/5!,  \ldots)$ are
$$
H_{2n+1}=  0; \quad
H_{2n}= (-1)^n 
 2^{2n(2n-1)} \prod_{k=1}^{2n-1} \frac {k!^2}{(2k+1)!^2}.
$$

\text{\rm (H15)}
The Hankel determinants of $(0, E_1/1!,0,  E_3/3!,0, E_5/5!,  \ldots)$ are
$$
H_{2n+1}=  0; \quad
H_0=1,\quad
H_{2n}= (-1)^n 
 2^{2(n-1)(2n-1)} \prod_{k=1}^{2n-2} \frac {k!^2}{(2k+1)!^2}.
$$

\text{\rm (H16)}
The Hankel determinant of $(E_1/1!, E_2/2!, E_3/3!, E_4/4!, \ldots)$ are
$$
H_0=1, \quad H_n =  \frac{(n-1)!}{2^{n-1}(2n-1)!} \prod_{k=1}^{n-2} \frac {k!^2}{(2k+1)!^2}.
$$

\text{\rm (H17)}
The Hankel determinants of $(E_0/0!, E_1/1!, E_2/2!, E_3/3!, \ldots)$ are
$$
H_0=1, \quad H_n = (-1)^{n(n-1)/2}  \frac{1}{2^{n-1}} \prod_{k=2}^{n-1} \frac {(k-1)!^2}{(2k-1)!^2}.
$$

\text{\rm (H18)}
The Hankel determinant of $(E_2/2!, E_3/3!, E_4/4!, \ldots)$ are
$$
H_n = (-1)^{n(n-1)/2}  \frac{1}{2^{n}} \prod_{k=2}^{n} \frac {(k-1)!^2}{(2k-1)!^2}.
$$

\text{\rm (H19)}
The Hankel determinants of $(E_3/3!, E_4/4!, E_5/5!, \ldots)$ are
$$
H_n = \frac {(n+1)(n+1)!}{2^n(2n+1)!} \prod_{k=1}^{n-1} \frac {k!^2 }{(2k+1)!^2}
$$

\text{\rm (H20)}
The Hankel determinants of $(E_4/4!, E_5/5!, E_6/6!, \ldots)$ are
$$
H_n = (-1)^{n(n-1)/2} \frac {(n+1)(n+2)(n^2+3n+1)}{2^{n+1}}
\prod_{k=1}^{n}\frac {k!^2}{(2k+1)!^2}.
$$

\text{\rm (H21)}
The Hankel determinants of $(E_3/3!,0, E_5/5!,  \ldots)$ are
\begin{align*}
	H_{2n}&= 2^{4n^2-1} \frac{(2n+2)!}{(4n+1)!} \prod_{k=1}^{2n-1} \frac {k!^2}{(2k+1)!^2},\\
	H_{2n+1}&= 2^{4n(n+1)} \frac{(2n+1)(2n+2)!}{(4n+3)!} \prod_{k=1}^{2n} \frac {k!^2}{(2k+1)!^2}.
\end{align*}

\text{\rm (H22)}
The Hankel determinants of $(E_3/3!, E_5/5!, E_7/7!,  \ldots)$ are
$$
H_{n}= 2^{n(2n-1)}  \prod_{k=1}^{2n-1} \frac {k!}{(2k+1)!}.
$$

\text{\rm (H23)}
The Hankel determinants of $(E_5/5!, E_7/7!,  \ldots)$ are
$$
H_{n}= 2^{n(2n+1)}(n+1)(2n+1)  \prod_{k=1}^{2n} \frac {k!}{(2k+1)!}.
$$

\text{\rm (H24)}
The Hankel determinants of $(E_7/7!, E_9/9!, \ldots)$ are
$$
H_{n}= 2^{n(2n+3)} (2n+1)(4n^2+10n+3)
\frac{(n+1)(n+2)(2n+3)}{3}
\prod_{k=1}^{2n+1} \frac {k!}{(2k+1)!}.
$$
\end{Theorem}


\section{New $q$-analog of the Euler numbers} \label{sec:q} 
There exist three kinds of $q$-analogs of the tangent and secant numbers, 
defined via (i) the $q$-sine and $q$-cosine functions introduced by Jackson \cite{Jackson1904, Foata2010Han, Andrews1978Gessel, Prodinger2011}; 
(ii) Lambert's continued fraction \eqref{eq:tan_Lambert} of $\tan(x)$ (see \cite{Fulmek2000, Prodinger2000});
(iii) The continued fractions \eqref{eq:E2n} and \eqref{eq:E2np1} of the ordinary generating functions of these numbers \cite{JosuatVerges2013Kim, Shin2010Zeng, Han1999RZeng, JosuatVerges2010}.

Let $[n]_q=1+q+q^2+\cdots +q^{n-1}, 
[n]!_q = [1]_q [2]_q \cdots [n]_q$ and 
$$\binom{n}{k}_q=\frac {[n]!_q}{[k]!_q [n-k]!_q}.$$
Version (iii) of the  $q$-secant and $q$-tangent numbers are
defined by
\begin{align}	
	\sum_{n\geq 0} \hat E_{2n}(q) x^{2n} &= 
	\cFrac {1}{1} - \cFrac{[1]_q^2x^2}{1}
	- \cFrac{[2]_q^2x^2}{1}- \cFrac{[3]_q^2x^2}{1} - \cdots \label{eq:qE2n} \\
	\sum_{n\geq 0} \hat E_{2n+1}(q) x^{2n+1} &= 
	\cFrac {x}{1} - \cFrac{[1]_q\cdot [2]_qx^2}{1}
	- \cFrac{[2]_q\cdot [3]_qx^2}{1}
- \cdots\label{eq:qE2np1}
\end{align}

In the same manner, we now use our theorem \ref{th:Enxn} to define a {\it new $q$-analog} of the Euler numbers as follows
\begin{equation}\label{eq:qEn}
	\sum_{n\geq 0} E_{n}(q) x^n =
	1+ \cFrac{x}{1-x } - \cFrac{\binom{2}{2}_qx^2}{1-[2]_qx } - 
	\cFrac{\binom{3}{2}_q x^2}{1-[3]_qx } 
	- \cFrac{\binom{4}{2}_q x^2}{1-[4]_qx } 
	-\cdots
\end{equation}
The general pattern for the coefficients $a_k$ and $b_k$ are: 
\begin{equation*}
	a_1=x,\quad      	a_k =-\binom{k}{2}_q x^2;           \qquad 
	b_0=1,\quad 	b_k = 1-[k]_qx .
\end{equation*}
The first values of $E_n(q)$ are listed below:
\begin{align*}
	E_0(q)&= E_1(q)=E_2(q)=1, \quad E_3(q)=2, \\
	E_4(q) &=  q + 4,  \\
	E_5(q) &=	 2q^2 + 5q + 9, \\
	E_6(q) &=  q^4 + 5q^3 + 14q^2 + 20q + 21.
\end{align*}
The specializations for $q=1,0,-1$ are
\begin{equation*}
	\begin{array}{*{13}c}
		n      &=& 0 & 1 & 2 & 3 & 4 & 5 & 6 & 7 & 8 & 9   \\
		\hline
		E_n(1)  &  =& 1 & 1  &1  &2 & 5  &16& 61& 272 & 1385& 7936  \\
		E_n(0)  &  =& 1 & 1  &1  &2 & 4  &9 & 21& 51 & 127 & 323  \\
		E_n(-1)  &  =& 1 & 1  &1  &2 & 3  &6& 11& 24 & 51& 122 \\
		\hline
\end{array}%
\end{equation*}
Of course $(E_n(1))_{n\geq 0}$ are just the Euler numbers.
Also, it is easy to see that $(E_{n}(0))_{n\geq 1}$ 
are the Motzkin numbers (see \cite[Proposition 5]{Flajolet1980},
\cite[p. 238]{Stanley1999EC2})
.
The most interesting case is $q=-1$. We have a non-trivial explicit formula for $E_n(-1)$, as stated next.\footnote{\,Bao-Xuan Zhu has informed me that
\eqref{eq:En_neg1_gf} is a consequence of a formula obtained by 
Paul Barry \cite[Prop. 8]{Barry2009}.
}

\begin{Theorem}\label{th:En_neg1}
	We have $E_0(-1)=1$ and
\begin{equation}\label{eq:En_neg1}
	E_n(-1)= \sum_{k=0}^{n-1} \binom{n-k-1}{k} k!
\end{equation}
or equivalently,
	\begin{equation}\label{eq:En_neg1_gf}
	\sum_{n\geq 0}   \sum_{k=0}^n \binom{n-k}{k} k! x^n 
	= \cFrac{1}{1-x} - \cFrac{x^2}{1} - \cFrac{x^2}{1-x} 
		- \cFrac{2x^2}{1} - \cFrac{2x^2}{1-x} 
		- \cdots
	\end{equation}
The general patterns for the new coefficients $a_j$ and $b_j$ are:
$$
a_1=1, \ 
	a_{2k} =-kx^2, \ 
	a_{2k+1} =-kx^2 ; \quad 
b_0=0, \ 
	b_{2k} = 1,  \ 
	b_{2k+1}= 1-x. 
$$
\end{Theorem}
\begin{proof}
	To prove identity \eqref{eq:En_neg1_gf}, we need to guess a unified
	property for 
	the following continued fraction with one more parameter $u$:
	\begin{equation*}
		F_u(x)
		= \cFrac{1}{1-x} - \cFrac{(u+1)x^2}{1} - \cFrac{(u+1)x^2}{1-x} 
		- \cFrac{(u+2)x^2}{1} - \cFrac{(u+2)x^2}{1-x} 
		- \cdots
	\end{equation*}
	It is clear that
	$F_0(x)$ is the continued fraction in \eqref{eq:En_neg1_gf}.
	We claim that $F_u(x)$ satisfies the following differential equation
\begin{equation*}
	(x-1)x^3 F_u'(x) -(x-1)(x-2)x^2 uF_u(x)^2  + (x-1)(2x^2+x-2) F_u(x) + (x-2)=0.
\end{equation*}
	Since we do not have efficient tools for guessing the above equation with the parameter $u$, our method consists of two steps:  first guessing formulas for specific values $u=0,1,2,3,\ldots$, then finding a unified pattern. This method was fully detailed with another example in \cite{Han2015NT}.

	To prove the above differential equation we let $S_u(x)$ be its left-hand side. 
	By using the following relation between $F_u(x)$ and $F_{u+1}(x)$
$$
	F_u(x)=\cfrac{1}{1-x-\cfrac{(u+1)x^2}{1-(u+1)x^2F_{u+1}(x)}},
$$
	we can express $S_u(x)$  in function of $F_{u+1}(x)$ and $F_{u+1}'(x)$. 
	After simplification we obtain
	\begin{equation}\label{eq:Su}
	S_u(x) = 
	\frac{(u + 1)^2x^4 S_{u+1}(x)  }{ 
	\left( (u + 1)(x - 1)x^2 F_{u+1}(x) +( 1-x- (u+1)x^2  )  \right)^2 }.
	\end{equation}
Since the constant term in $x$ of the denominator in the above fraction is equal to~$1$, applying relation \eqref{eq:Su} iteratively yields that $S_u(x)=0$.
Notice that we cannot prove that $S_u(x)=0$ by induction, since
	the base case for induction would be $S_\infty(x)=0$. This is impossible to prove because $F_\infty(x)$ does not exist.
Now take $u=0$. We have $S_0(x)=0$, or
	$$
	(x-1)x^3 F_0'(x)  + (x-1)(2x^2+x-2) F_0(x) + (x-2)=0.
	$$
	Write $\alpha_n:=E_{n+1}(-1)$ for short. Comparing the coefficient of $x^n$ in the above equation, we know that the $\alpha_n$'s satisfy the recurrence relation
	\begin{equation}\label{eq:Ennegrec}
	2\alpha_n = 3\alpha_{n-1} + (n-1)\alpha_{n-2} - (n-1)\alpha_{n-3},
	\end{equation}
	with initial values $\alpha_0=\alpha_1=1$ and $\alpha_2=2$.
	Finally we can prove \eqref{eq:En_neg1} by Zeilberger's 
algorithm \cite{Petkovsek1996etal}.
\end{proof}

{\it Remark.} 
Under the sequence A122852 in the OEIS \cite{OEIS:A122852},
formula \eqref{eq:En_neg1} 
is given by Paul Barry without proof and reference;
as well as recurrence \eqref{eq:Ennegrec} is stated by R. J. Mathar as a conjecture. The above proofs of \eqref{eq:En_neg1} and \eqref{eq:Ennegrec} are unexpectedly non-trivial. 

\smallskip

Applying Heilermann's formula \eqref{eq:detH} 
to the $J$-fraction \eqref{eq:En_neg1_gf} we 
obtain
$$
\det\left(\sum_{k=0}^{i+j}\binom{i+j-k}{k} k!\right)_{i,j=0}^{2n-1}
=\prod_{k=1}^n (k-1)!^3 k!
$$
and
$$
\det\left(\sum_{k=0}^{i+j}\binom{i+j-k}{k} k!\right)_{i,j=0}^{2n}
=\prod_{k=1}^n (k-1)! k!^3.
$$

\medskip

The continued fractions \eqref{eq:qE2n} and \eqref{eq:qE2np1} 
lead to several combinatorial interpretations of the polynomials 
$\hat E_{2n}(q)$ and $\hat E_{2n+1}(q)$, see
\cite{JosuatVerges2013Kim, Shin2010Zeng, Han1999RZeng, JosuatVerges2010}.
It would be interesting to find a combinatorial model for the new $q$-Euler numbers $E_q(n)$.\footnote{\, This combinatorial model has been found by
Qiong Qiong Pan and Jiang Zeng \cite{Pan2019ZengEuler}.}

\medskip

{\bf Acknowledgments}. The author would like to thank 
Alan Sokal, Bao-Xuan Zhu, Yan Zhuang and an anonymous referee 
for knowledgeable remarks and helpful suggestions.


\medskip

\bibliographystyle{plain}




\end{document}